\documentclass{amsart}

\usepackage{amssymb,amsmath,amsthm,amsfonts}
\usepackage{nicefrac,mathtools,enumerate}

\newcommand{\alg}{\mathbf}
\newcommand{\class}{\mathsf}
\newcommand{\logic}{\mathrm}
\newcommand{\lang}{\mathcal}

\newcommand{\set}[2]{\{ #1 \mid #2 \}}
\newcommand{\pair}[2]{\langle #1, #2 \rangle}
\newcommand{\tuple}{\overline}

\newcommand{\iso}{\cong}
\newcommand{\equals}{\thickapprox}

\newcommand{\assign}{\mathrel{:=}}

\newcommand{\boldTheta}{\mathbf{\Theta}}
\newcommand{\boldPsi}{\mathbf{\Psi}}

\newcommand\onehalf{\nicefrac12}

\newcommand{\Kthree}{\alg{K}_{\alg{3}}}
\newcommand{\DMfour}{\alg{DM}_{\alg{4}}}
\newcommand{\WKthree}{\alg{WK}_{\alg{3}}}

\DeclareMathOperator{\Alg}{Alg}
\DeclareMathOperator{\Cg}{Cg}
\DeclareMathOperator{\Fg}{Fg}
\DeclareMathOperator{\Fi}{Fi}
\newcommand{\Fm}{\alg{Fm}}

\newcommand{\Leibniz}{\Omega}

\newcommand{\Tarski}{\widetilde{\Omega}}

\newcommand{\ddtset}{I}
\newcommand{\ddtfamily}{\Phi}
\newcommand{\ddtformula}{\alpha}

\DeclareMathOperator{\At}{At}
\DeclareMathOperator{\Cn}{Cn}
\DeclareMathOperator{\Con}{Con}
\DeclareMathOperator{\Ker}{Ker}

\DeclareMathOperator{\Var}{Var}

\newcommand{\AlgH}{\mathbb{H}}
\newcommand{\AlgI}{\mathbb{I}}
\newcommand{\AlgS}{\mathbb{S}}
\newcommand{\AlgP}{\mathbb{P}}
\newcommand{\AlgPU}{\mathbb{P}_{\mathrm{U}}}

\newcommand{\CL}{\logic{CL}}
\newcommand{\KL}{\logic{KL}}
\newcommand{\LP}{\logic{LP}}
\newcommand{\PWK}{\logic{PWK}}
\newcommand{\ETL}{\logic{ETL}}

\newcommand{\IL}{\logic{IL}}
\newcommand{\LocK}{\logic{K}^{\ell}}
\newcommand{\GlobK}{\logic{K}^{\mathrm{g}}}

\DeclareSymbolFont{forpolishl}{T1}{cmr}{m}{n}
\DeclareMathSymbol{\mathrmL}{0}{forpolishl}{'212}
\newcommand{\Luk}{\mathrmL}

\newcommand{\ILbc}{\IL_{\circ\bullet}}
\newcommand{\ESL}{\logic{L}_{\rightarrow}}
\newcommand{\IdL}{\logic{Id}}

\newcommand{\MV}{\class{MV}}
\newcommand{\HA}{\class{HA}}
\newcommand{\BAO}{\class{BAO}}
\newcommand{\KA}{\class{KA}}

\newcommand{\GIB}{\class{GIB}}

\newtheorem{theorem}{Theorem}[section]
\newtheorem{lemma}[theorem]{Lemma}

\newtheorem{corollary}[theorem]{Corollary}
\newtheorem{fact}[theorem]{Fact}
\newtheorem{example}[theorem]{Example}

\theoremstyle{definition}
\newtheorem{definition}[theorem]{Definition}

\title{Equational definitions of logical filters}
\author{Michele Pra Baldi}
\address{Universit\`{a} di Padova, Dipartimento di Filosofia, Sociologia, Pedagogia e Psicologia Applicata - FISPPA}
\email{michele.prabaldi@unipd.it}
\author{Adam P\v{r}enosil}
\address{Universitat de Barcelona, Departament de Filosofia}
\email{adam.prenosil@gmail.com}
\thanks{The work of the second author was funded by the grant 2021 BP 00212 of the grant agency AGAUR of the Generalitat de Catalunya.}

\begin{document}

\begin{abstract}
  A finitary propositional logic can be given an algebraic reading in two different ways: by translating formulas into equations and logical rules into quasi-equations, or by translating logical rules directly into equations. The former type of algebraic interpretation has been extensively studied and underlies the theory of algebraization. Here we shall develop a systematic theory of the latter type of algebraic interpretation. More precisely, we consider a semantic form of this property which we call the equational definability of compact filters (EDCF). Paralleling the well-studied hierarchy of variants of the deduction--detachment theorem (DDT), this property also comes in local, parametrized, and parametrized local variants. Our main results give a semantic characterization of each of these variants of the EDCF in a spirit similar to the existing characterizations of the DDT. While the EDCF hierarchy and the DDT hierarchy coincide for algebraizable logics, part of the interest of the EDCF stems from the fact it is often enjoyed even by logics which are not well-behaved in terms of other existing classifications in algebraic logic.
\end{abstract}

\maketitle

\section{Introduction}

  A central concern of algebraic logic, particularly of its subfield known as abstract algebraic logic (AAL), is to systematically relate the consequence relation of a propositional logic $\logic{L}$ with the equational consequence relation of some class of algebras~\cite{font16}. (Throughout the paper, by a logic we mean a \emph{finitary} structural consequence relation between a set of premises and a single conclusion.) This class of algebras is typically the \emph{algebraic counterpart} of $\logic{L}$: a class of algebras $\Alg \logic{L}$ canonically associated to any logic. For example, the consequence relation $\IL$ of intuitionistic logic is linked to the equational consequence relation of the class $\HA$ of Heyting algebras by the following equivalence:
\begin{align*}
  \gamma_{1}, \dots, \gamma_{n} \vdash_{\IL} \varphi \iff \gamma_{1} \equals 1, \dots, \gamma_{n} \equals 1 \vDash_{\HA} \varphi \equals 1.
\end{align*}
  Here the right-hand side of the equivalence states that the corresponding quasi-equation holds in all Heyting algebras, i.e.\ instead we may equally well write
\begin{align*}
  \HA \vDash (\gamma_{1} \equals 1 ~ \& ~ \dots ~ \& ~ \gamma_{n} \equals 1) \implies \varphi \equals 1.
\end{align*}
  Such an equivalence is called an \emph{equational completeness theorem}~\cite{moraschini22}. Indeed, this equivalence is a \emph{standard} equational completeness theorem: the class of algebras on the right-hand side ($\HA$) is the algebraic counterpart of $\IL$.\footnote{The recent work of Moraschini~\cite{moraschini22} showed that almost all logics admit \emph{some} equational completeness theorem, but conversely also that, absent any constraints on the class of algebras, such completeness theorems are of little help in the study of a given logic.}

  However, this not the only way to provide the consequence relation of intuitionistic logic with an algebraic reading. We can also translate intuitionistic consequence into the equational theory of Heyting algebras as follows:
\begin{align*}
  \gamma_{1}, \dots, \gamma_{n} \vdash_{\IL} \varphi \iff \HA \vDash \gamma_{1} \wedge \dots \wedge \gamma_{n} \leq \varphi.
\end{align*}
  Rather than individual formulas getting translated into equations, here the entire consequence $\gamma_{1}, \dots, \gamma_{n} \vdash \varphi$ is translated into an equation. Accordingly, we shall call this kind of equivalence a \emph{(global) equational definition of consequence}, and we say that a logic with such an equational definition of consequence \emph{has the (global) EDC}. In particular, the above equivalence is a \emph{standard} equational definition of consequence, again in the sense that the right-hand side involves the algebraic counterpart of $\IL$.

  Both of these algebraic readings of logical consequence are available for intuitionistic logic, where it is in fact easy to derive one from the other. They are more generally, they are available given any class of algebras $\class{K}$ with a unital meet semilattice reduct. In that case, the logic defined by the equational completeness theorem of the above form is called the \emph{assertional logic} of $\class{K}$ and the logic defined by the equational definition of consequence of the above form is called the \emph{logic of order} of $\class{K}$. Beyond the case of intuitionistic logic, the assertional logic and the logic of order are typically different.

  Consider\footnote{See~\cite[Section~1.5]{blackburn+derijke+venema01} for the definition of these consequence relations and~\cite[p.~476]{font16} for a summary of the basic properties of these logics from the point of view of AAL, including those stated here without proof.} for instance the local and global variants of basic modal logic, $\LocK$ and $\GlobK$. These are two distinct consequence relations: $x \vdash_{\GlobK} \Box x$ but $x \nvdash_{\LocK} \Box x$. The algebraic counterpart of both of these logics is the variety (equational class) $\BAO$ of Boolean algebras with an operator (a unary operation $\Box$ which satisfies the equations $\Box (x \wedge y) \equals \Box x \wedge \Box y$ and $\Box 1 \equals 1$). The logic $\LocK$ is the logic of order of $\BAO$, and a such has a standard global EDC, namely:
\begin{align*}
  \gamma_{1}, \dots, \gamma_{n} \vdash_{\LocK} \varphi \iff \BAO \vDash \gamma_{1} \wedge \dots \wedge \gamma_{n} \leq \varphi.
\end{align*}
  In contrast, $\LocK$ does not have any standard equational completeness theorem~\cite[Corollary~9.7]{moraschini22}. The opposite situation holds for $\GlobK$. This is the assertional logic of $\BAO$, and as such it has a standard equational completeness theorem of the following form:
\begin{align*}
  \gamma_{1}, \dots, \gamma_{n} \vdash_{\GlobK} \varphi \iff \gamma_{1} \equals 1, \dots, \gamma_{n} \equals 1 \vDash_{\BAO} \varphi \equals 1,
\end{align*}
  but it does not have a standard global EDC (Example~\ref{example: global k does not have edc}).

  Although the logic $\GlobK$ does not admit a standard global EDC, it does admit the following more general kind of standard EDC:
\begin{align*}
  \gamma_{1}, \dots, \gamma_{n} \vdash_{\GlobK} \varphi \iff \BAO \vDash \Box_{k} (\gamma_{1} \wedge \dots \wedge \gamma_{n}) \leq \varphi \text{ for some } k \in \omega,
\end{align*}
  where we use the notation $\Box_{0} x \assign x$ and $\Box_{i+1} x \assign x \wedge \Box \Box_{i} x$. Similarly, the infinite-valued {\L}ukasiewicz logic $\Luk$, whose algebraic counterpart is the variety $\MV$ of MV-algebras, does not admit a standard global EDC (Example~\ref{example: luk does not have edc}), but it admits the following more general kind of standard EDC:
\begin{align*}
  \gamma_{1}, \dots, \gamma_{n} \vdash_{\Luk} \varphi \iff \MV \vDash (\gamma_{1} \wedge \dots \wedge \gamma_{n})^{k} \leq \varphi \text{ for some } k \in \omega,
\end{align*}
  where $\varphi^{0} \assign 1$ and $\varphi^{i+1} \assign \varphi \odot \varphi^{i}$. Such an equivalence, where the right-hand side states that at least one equation in a certain family of equations holds, will be called a \emph{local} equational definition of consequence.

  Other examples show that even if both an equational completeness theorem and an equational definition of consequence are available for a given logic, they may provide two rather different connections to the same class of algebras. Consider for instance the strong three-valued Kleene logic $\KL$ and the three-valued Logic of Paradox $\LP$.\footnote{See \cite[Sections 3 and 5]{albuquerque+prenosil+rivieccio17} and \cite[Section 4]{prenosil23} for a discussion of these logics from the point of view of AAL.} These logics are typically studied in the signature $\{ \wedge, \vee, \neg \}$, but for the sake of simplicity let us throughout this paper add the constants $0$ and $1$ to their signature. The algebraic counterpart of both of these logics is the class $\KA$ of Kleene algebras, i.e.\ bounded distributive lattices with an order-inverting involution $\neg$ which satisfy the inequality $x \wedge \neg x \leq y \vee \neg y$. The logic $\KL$ has a standard equational completeness theorem of the form
\begin{align*}
  \gamma_{1}, \dots, \gamma_{n} \vdash_{\KL} \varphi & \iff \gamma_{1} \equals 1, \dots, \gamma_{n} \equals 1 \vDash_{\KA} \varphi \equals 1,
\end{align*}
  while $\LP$ has a standard equational completeness theorem of the form
\begin{align*}
  \gamma_{1}, \dots, \gamma_{n} \vdash_{\LP} \varphi & \iff \neg \gamma_{1} \leq \gamma_{1}, \dots, \neg \gamma_{n} \leq \gamma_{n} \vDash_{\KA} \neg \varphi \leq \varphi.
\end{align*}
  These logics also have standard equational definitions of consequence:
\begin{align*}
  \gamma_{1}, \dots, \gamma_{n} \vdash_{\KL} \varphi & \iff \KA \vDash \gamma \leq \neg \gamma \vee \varphi \text{ for } \gamma \assign \gamma_{1} \wedge \dots \wedge \gamma_{n}, \\
  \gamma_{1}, \dots, \gamma_{n} \vdash_{\LP} \varphi & \iff \KA \vDash \gamma \wedge \neg \varphi \leq \varphi \text{ for } \gamma \assign \gamma_{1} \wedge \dots \wedge \gamma_{n}.
\end{align*}
  In other words, the consequence relations of $\KL$ and $\LP$ may be identified either with a fragment of the quasi-equational theory of Kleene algebras through the equational completeness theorem or with a (different) fragment of the equational theory of Kleene algebras through the equational definition of consequence.

  The theory of equational completeness theorems forms a core part of AAL. Indeed, one can say that AAL was developed as an attempt to provide an appropriate framework for studying such theorems. In contrast, no corresponding systematic theory of the equational definability of consequence exists, not even in a rudimentary form. Only the special case of logics of order has been studied in detail~\cite{jansana06}. The goal of this paper is to develop the foundations of such a theory. Our main concern will be to establish the semantic correlates of the equational definability properties of the type considered above.

  To be more precise, rather than considering the equational definability of consequence, we shall study a stronger form of this property: the equational definability of \emph{filter generation}. Before we can state the definition of this property, we will need to briefly review some basic concepts of AAL.

  Just as an \emph{$\logic{L}$-theory} is a set of formulas closed under the consequences valid in a logic $\logic{L}$, an \emph{$\logic{L}$-filter} on an algebra $\alg{A}$ is a subset $F$ of $\alg{A}$ which is closed, in a natural sense, under the consequences valid in $\logic{L}$. For example, if $\alg{A}$ is a Boolean algebra with an operator, then the filters of $\LocK$ on $\alg{A}$ are the lattice filters, while the filters of $\GlobK$ on $\alg{A}$ are the lattice filters $F$ closed under the necessitation rule $x \vdash \Box x$, which is interpreted in $\alg{A}$ as $a \in F \implies \Box a \in F$. $\logic{L}$-theories are precisely the $\logic{L}$-filters on the algebra of formulas (the absolutely free algebra) in the signature of $\logic{L}$ over a given infinite set of variables. In other words, the notion of an $\logic{L}$-filter extends that of an $\logic{L}$-theory to arbitrary algebras.

  The deductive closure $\Cn_{\logic{L}} (\gamma_{1}, \dots, \gamma_{n})$ of a finite set of formulas $\{ \gamma_{1}, \dots, \gamma_{n} \}$, which is the smallest $\logic{L}$-theory containing $\{ \gamma_{1}, \dots, \gamma_{n} \}$, then generalizes to the $\logic{L}$-filter $\Fg^{\alg{A}}_{\logic{L}}(a_{1}, \dots, a_{n})$ generated by $\{ a_{1}, \dots, a_{n} \} \subseteq \alg{A}$, which is the smallest $\logic{L}$-filter on $\alg{A}$ containing $\{ a_{1}, \dots, a_{n} \}$. Instead of considering the equational definability of the relation $\varphi \in \Cn_{\logic{L}} (\gamma_{1}, \dots, \gamma_{n})$, which is merely the relation $\gamma_{1}, \dots, \gamma_{n} \vdash_{\logic{L}} \varphi$ under a different notation, we can now ask whether the relation $b \in \Fg^{\alg{A}}_{\logic{L}} (a_{1}, \dots, a_{n})$ is equationally definable for each $\alg{A} \in \Alg \logic{L}$. We call this property the \emph{equational definability of compact filters} (EDCF).\footnote{The name comes from the fact that finitely generated $\logic{L}$-filters, i.e.\ $\logic{L}$-filters of the form $\Fg^{\alg{A}}_{\logic{L}}(a_{1}, \dots, a_{n})$, are exactly the compact elements of the lattice of all $\logic{L}$-filters on $\alg{A}$. We shall use EDCF as an abbreviation for either \emph{equational definability of compact filters} or \emph{equational definition of compact filters}, depending on the context.} 

  The equational definitions of consequence in the examples that we have seen so far all extend straightforwardly to equational definitions of compact filters. Local modal logic $\LocK$ enjoys the following form of the (global) EDCF: for each $\alg{A} \in \BAO = \Alg \LocK$ and each $a_{1}, \dots, a_{n}, b \in \alg{A}$
\begin{align*}
  b \in \Fg^{\alg{A}}_{\LocK} (a_{1}, \dots, a_{n}) \iff a_{1} \wedge \dots \wedge a_{n} \leq b.
\end{align*}
  The logics $\KL$ and $\LP$ also have the (global) EDCF: for each $\alg{A} \in \KA = \Alg \KL = \Alg \LP$ and each $a_{1}, \dots, a_{n}, b \in \alg{A}$
\begin{align*}
  b \in \Fg^{\alg{A}}_{\KL} (a_{1}, \dots, a_{n}) \iff a \leq \neg a \vee b \text{ for } a \assign a_{1} \wedge \dots \wedge a_{n}, \\
  b \in \Fg^{\alg{A}}_{\LP} (a_{1}, \dots, a_{n}) \iff a \wedge \neg b \leq b \text{ for } a \assign a_{1} \wedge \dots \wedge a_{n}.
\end{align*}
  The local EDC of global modal logic $\GlobK$ extends to the following \emph{local} form of the EDCF:  for each $\alg{A} \in \BAO = \Alg \GlobK$ and each $a_{1}, \dots, a_{n}, b \in \alg{A}$
\begin{align*}
  b \in \Fg^{\alg{A}}_{\GlobK} (a_{1}, \dots, a_{n}) \iff \Box_{k}(a_{1} \wedge \dots \wedge a_{n}) \leq b \text{ for some } k \in \omega.
\end{align*}
  Similarly, compact filter generation in infinite-valued {\L}ukasiewicz logic $\Luk$ is described by the following local form of the EDCF: for each $\alg{A} \in \MV = \Alg \Luk$ and each $a_{1}, \dots, a_{n}, b \in \alg{A}$
\begin{align*}
  b \in \Fg^{\alg{A}}_{\Luk} (a_{1}, \dots, a_{n}) \iff (a_{1} \wedge \dots \wedge a_{n})^{k} \leq b \text{ for some } k \in \omega,
\end{align*}
  or equivalently
\begin{align*}
  b \in \Fg^{\alg{A}}_{\Luk} (a_{1}, \dots, a_{n}) \iff (a_{1} \odot \dots \odot a_{n})^{k} \leq b \text{ for some } k \in \omega.
\end{align*}
  While the EDCF is a stronger property than the standard EDC, we have not found any examples of logics with the standard EDC but not the EDCF ``in the wild'' and had to resort to constructing an artificial example in order to separate the two properties (see Example~\ref{example: edc without edcf}).

  The main contribution of this paper is to provide equivalent characterizations of the local and global EDCF (Theorems~\ref{thm: ledcf} and \ref{thm: baby edcf}) and of the parametrized forms of the local and global EDCF (Theorems~\ref{thm: pledcf} and \ref{thm: baby pedcf}) which can readily be used to prove that a logic does not admit any EDCF of the appropriate form. These theorems are very much in the spirit of the semantic characterizations of the corresponding forms of the so-called Deduction--Detachment Theorem, or DDT (see~\cite{blok+pigozzi01} and~\cite[Chapter~2]{czelakowski01} for a comprehensive account of the DDT). Indeed, there is a close connection between the established hierarchy of local and global forms of the DDT and the new hierarchy of local and global forms of the EDCF (with or without parameters).

  The DDT hierarchy and the EDCF hierarchy coincide for algebraizable logics whose algebraic counterpart is a quasivariety (Section~\ref{sec: edcf and ddt}). Such logics are the best-behaved among the major class of logics studied in AAL. Beyond the exclusive club of algebraizable logics, the two hierarchies diverge.

  The base level (i.e.\ the parametrized local level) of the DDT hierarchy consists of the so-called protoalgebraic logics, while the corresponding level of the EDCF hierarchy contains all logics. In full generality the two hierarchies are independent of each other except at this base level. However, in practice it only takes a very modest assumption to derive a given form of the EDCF from the corresponding form of the DDT (Fact~\ref{fact: ddt implies edcf}). Namely, it suffices to assume that the smallest filter on any algebra has an equational definition of the appropriate form (typically $1 \wedge x \equals 1$ or $1 \vee x \equals x$ or $1 \equals x$). That said, a logic which has the DDT but not the EDCF can be constructed (Example~\ref{example: ddt but not edc}).

  Part of the interest of equational definitions of compact filters stems from the fact that many logics which are not particularly well-behaved when viewed through the lens of other existing categories of abstract algebraic logic (on account of lacking a well-behaved implication connective) nonetheless have an EDCF. For instance, the logic $\KL$ is almost maximally ill-behaved: it does not have a standard equational completeness theorem (hence it is not truth-equational), it is not proto\-algebraic (hence it is not equivalential, much less algebraizable, and it does not have even a weak form of the deduction theorem), and it is not selfextensional. The connection between $\KL$ and its algebraic counterpart should therefore be very weak. Yet the logic has a global EDCF, unlike many logics which are much better behaved with respect to the existing categories of abstract algebraic logic, such as the global modal logic $\GlobK$, infinite-valued {\L}ukasiewicz logic, or the Full Lambek calculus.

  In addition, we consider the question of when the EDCF can be finitized in the sense of reducing to finitely many equations in the global case and to a finite family of finitely many equations in the local case (Theorem~\ref{thm: finitary edcf}). The finitized forms of our characterizations theorems, at least for the parametrized and for the local EDCF, are in fact straightforward consequences of the model-theoretic results of Campercholi and Vaggione~\cite{campercholi+vaggione16} concerning the definability of relations by first-order formulas of various types. However, note that many prominent examples of logics (such as $\GlobK$ and $\Luk$) have a non-finitizable local EDCF.

  The paper is structured as follows. In Section~\ref{sec: preliminaries} we review the basic notions of AAL which we will need to rely on throughout the paper. In Sections~\ref{sec: pledcf and pedcf} and~\ref{sec: ledcf and edcf} we provide a number of equivalent conditions for the local and global EDCF in their parametrized and parameter-free forms, respectively. In Section~\ref{sec: finitary edcf} we then discuss the finitary forms of the EDCF. Section~\ref{sec: leibniz} rreviews some basic classes of logics which arise in AAL. It provides the required preliminaries for Section~\ref{sec: edcf and ddt}, which discusses the relationship between the corresponding forms of the EDCF and the DDT, proving in particular that for algebraizable logics they coincide. Finally, in Section~\ref{sec: edc} we consider the equational definability of consequence, using $\GlobK$ and $\Luk$ as examples on which we show how to prove that a logic fails to have an equationally definable consequence relation.

\section{Preliminaries}
\label{sec: preliminaries}

  This preliminary section contains a review of the basic notions of AAL which will be used throughout the paper. The reader seeking a more detailed and leisurely introduction to AAL is advised to consult the textbook~\cite{font16}. 

  For the purposes of this paper, a \emph{logic} is a \emph{finitary} structural consequence relation in some given set of variables and a given algebraic signature. In more detail, we fix an infinite set of variables and a signature consisting of a set of function symbols with some specified finite arities. By algebras we shall always mean algebras in this given signature. The \emph{formula algebra} is the absolutely free algebra $\Fm$ generated in this algebraic signature by the given set of variables.

  A \emph{(finitary) rule} is a pair consisting of a finite set of formulas (the \emph{premises} of the ruel) and a single formula (the \emph{conclusion} of the rule), written as $\gamma_{1}, \dots, \gamma_{n} \vdash \varphi$, A logic $\logic{L}$ is then a set of finitary rules, where $\gamma_{1}, \dots, \gamma_{n} \vdash \varphi \in \logic{L}$ will be written as $\gamma_{1}, \dots, \gamma_{n} \vdash_{\logic{L}} \varphi$, such that
\begin{enumerate}[(i)]
\item $\gamma_{1}, \dots, \gamma_{n} \vdash_{\logic{L}} \varphi$ holds for $\varphi \in \{ \gamma_{1}, \dots, \gamma_{n} \}$,
\item if $\gamma_{1}, \dots, \gamma_{n} \vdash_{\logic{L}} \psi$ and $\psi, \delta_{1}, \dots, \delta_{k} \vdash_{\logic{L}} \varphi$, then $\gamma_{1}, \dots, \gamma_{n}, \delta_{1}, \dots, \delta_{k} \vdash_{\logic{L}} \varphi$,
\item if $\gamma_{1}, \dots, \gamma_{n} \vdash_{\logic{L}} \varphi$, then $\sigma(\gamma_{1}), \dots, \sigma(\gamma_{n}) \vdash_{\logic{L}} \sigma(\varphi)$ for each substitution $\sigma$.
\end{enumerate}
  The rules which belong to $\logic{L}$ are said to be \emph{valid} in $\logic{L}$ or to \emph{hold} in $\logic{L}$.

  Throughout the paper, $\logic{L}$ will denote a logic. A logic $\logic{L}$ is \emph{trivial} if it validates every rule. A \emph{theory} of $\logic{L}$, or an \emph{$\logic{L}$-theory}, is a set of formulas $T$ closed under $\logic{L}$-consequence: if $\gamma_{1}, \dots, \gamma_{n} \vdash_{\logic{L}} \varphi$ and $\gamma_{1}, \dots, \gamma_{n} \in T$, then $\varphi \in T$. We write $\Gamma \vdash_{\logic{L}} \Delta$ as shorthand for the claim that $\Gamma \vdash_{\logic{L}} \delta$ for each $\delta \in \Delta$.

  A \emph{(logical) matrix} is a structure $\pair{\alg{A}}{F}$ consisting of an algebra $\alg{A}$ and a subset $F \subseteq \alg{A}$. A finitary rule $\gamma_{1}, \dots, \gamma_{n} \vdash \varphi$ is said to be \emph{valid} in a matrix $\pair{\alg{A}}{F}$ if $h(\gamma_{1}), \dots, h(\gamma_{n}) \in F$ imply $h(\varphi) \in F$ for each homomorphism $h\colon \Fm \to \alg{A}$. An \emph{$\logic{L}$-filter} on an algebra $\alg{A}$ is a set $F \subseteq \alg{A}$ such that each rule valid in $\logic{L}$ is also valid in $\pair{\alg{A}}{F}$. In that case, we also say that $\pair{\alg{A}}{F}$ is a \emph{model} of $\logic{L}$. Informally speaking, an $\logic{L}$-filter on $\alg{A}$ is a subset of $\alg{A}$ closed under all rules valid in $\logic{L}$, as interpreted in~$\alg{A}$. The logic \emph{determined by} a class of matrices $\class{K}$ is the logic consisting precisely of the rules valid in each matrix in $\class{K}$.

  The $\logic{L}$-filters on $\alg{A}$ form an algebraic lattice $\Fi_{\logic{L}} \alg{A}$, where meets are inter\-sections and directed joins are directed unions. The smallest $\logic{L}$-filter which extends a given set $X \subseteq \alg{A}$ is called the $\logic{L}$-filter \emph{generated} by $X$ and denoted by $\Fg^{\alg{A}}_{\logic{L}} X$. The compact elements of the lattice $\Fi_{\logic{L}} \alg{A}$ are precisely the finitely generated $\logic{L}$-filters, i.e.\ $\logic{L}$-filters of the form $\Fg^{\alg{A}}_{\logic{L}} (a_{1}, \dots, a_{n}) \assign \Fg^{\alg{A}}_{\logic{L}} \{ a_{1}, \dots, a_{n} \}$ for some $a_{1}, \dots, a_{n} \in \alg{A}$. It is important to keep in mind that here and in other similar contexts we allow for $n \assign 0$. In that case the notation $\Fg^{\alg{A}}_{\logic{L}} (a_{1}, \dots, a_{n})$ is to be understood as $\Fg^{\alg{A}}_{\logic{L}} \emptyset$.

  A congruence $\theta$ of an algebra $\alg{A}$ is \emph{compatible} with a subset $F$ of $\alg{A}$ in case
\begin{align*}
  a \in F \text{ and } \pair{a}{b} \in \theta \implies b \in F.
\end{align*}
  Informally speaking, this means that the congruence $\theta$ does not cut across $F$. If $F$ is an $\logic{L}$-filter on $\alg{A}$ and $\theta$ is a congruence on $\alg{A}$ compatible with $F$, then $F / \theta \assign \set{a / \theta \in \alg{A} / \theta}{a \in F}$ is an $\logic{L}$-filter on $\alg{A} / \theta$. If $\pi_{\theta}\colon \alg{A} \to \alg{A} / \theta$ is the quotient map, then of course $F = \pi_{\theta}^{-1}[F / \theta]$.

  Using the notion of compatibility, we can introduce the \emph{algebraic counterpart} of a logic $\logic{L}$. This class, denoted by $\Alg \logic{L}$, consists of those algebras $\alg{A}$ where the only congruence compatible with all $\logic{L}$-filters on $\alg{A}$ is the identity congruence~$\Delta_{\alg{A}}$. Informally speaking, this means that $\alg{A}$ does not contain any logically redundant information: we cannot quotient $\alg{A}$ any further without thereby destroying some of the $\logic{L}$-filters on $\alg{A}$.

  The algebraic counterpart of a logic $\Alg \logic{L}$ is always a \emph{subdirect class} (a class of algebras closed under isomorphic images and subdirect products). Often but not always, the class $\Alg \logic{L}$ is closed under subalgebras, making it a \emph{prevariety} (a class of algebras closed under isomorphic images, subalgebras, and products), or even under subalgebras and ultraproducts, making it a \emph{quasivariety} (a prevariety closed under ultraproducts). In the best of cases, the algebraic counterpart of a logic is a \emph{variety} (a prevariety closed closed under homomorphic images).

\begin{fact} \label{fact: alg l variety}
  Let $\logic{L}$ be the logic determined by a class of matrices whose algebraic reducts lie in a variety $\class{K}$. Then $\Alg \logic{L} \subseteq \class{K}$.
\end{fact}

\begin{proof}
  This follows from Theorem~5.76 and Proposition 5.79 of \cite{font16}.
\end{proof}

 Given a subdirect class $\class{K}$, a \emph{$\class{K}$-congruence} on an algebra $\alg{A}$ is a congruence $\theta$ on $\alg{A}$ such that $\alg{A} / \theta \in \class{K}$. The $\class{K}$-congruences on each algebra $\alg{A}$ form a complete lattice $\Con_{\class{K}} \alg{A}$, where meets are intersections. Tthe $\class{K}$-congruence \emph{generated} by a set of pairs $X \subseteq \alg{A}^{2}$ is the smallest $\class{K}$-congruence on $\alg{A}$ which extends $X$. This $\class{K}$-congruence is denoted by $\Cg^{\alg{A}}_{\class{K}} X$. In particular, the smallest $\class{K}$-congruence on $\alg{A}$ is $\Cg^{\alg{A}}_{\class{K}} \emptyset$. This smallest $\class{K}$-congruence will also be denoted by $\theta^{\alg{A}}_{\class{K}}$ here. Observe that for every algebra $\alg{A}$ there is a largest congruence compatible with all $\logic{L}$-filters on $\alg{A}$, namely the congruence $\theta^{\alg{A}}_{\Alg \logic{L}}$.

\section{Parametrized and parametrized local EDCF}
\label{sec: pledcf and pedcf}

  We now introduce the most general type of the EDCF to be considered here, namely the parametrized local EDCF. Other types of the EDCF will be defined as restrictions of the parametrized local EDCF. We then show that every logic in fact admits this form of the EDCF, and describe those logics which admit a parametrized EDCF. While this shows that the parametrized local EDCF is a trivial property of logics in the sense that every logic has one, keep in mind that knowing that a logic admits a parametrized local EDCF of a particular shape may well provide non-trivial information about the logic.

\begin{definition}
  A \emph{parametrized local equational definition of $n$-generated $\logic{L}$-filters} on a class of algebras $\class{K}$ for $n \in \omega$ is a family of sets equations $\boldPsi_{n}(x_{1}, \dots, x_{n}, y, \tuple{z})$, where the variables $x_{1}, \dots, x_{n}, y, \tuple{z}$ are distinct and $\tuple{z}$ is a possibly infinite tuple of variables (which we call \emph{parameters}), such that for each algebra $\alg{A} \in \class{K}$ and each $a_{1}, \dots, a_{n}, b \in \alg{A}$
\begin{align*}
  b \in \Fg^{\alg{A}}_{\logic{L}}(a_{1}, \dots, a_{n}) & \iff \alg{A} \vDash \boldTheta_{n}(a_{1}, \dots, a_{n}, b, \tuple{c}) \text{ for some } \boldTheta_{n} \in \boldPsi_{n}, \tuple{c} \in \alg{A}.
\end{align*}
  In a \emph{local definition} of $n$-generated filters each set $\boldTheta_{n} \in \boldPsi_{n}$ moreover has the form $\boldTheta_{n}(x_{1}, \dots, x_{n}, y)$, i.e.\ no parameters occur in~$\boldTheta_{n}$. A \emph{parametrized global definition} of $n$-generated filters is a set of equations $\boldTheta_{n}(x_{1}, \dots, x_{n}, y, \tuple{z})$ such that $\boldPsi_{n} \assign \{ \boldTheta_{n} \}$ is a parametrized local equational definition of $n$-generated filters. In a \emph{global equational definition} of $n$-generated filters this set $\boldTheta_{n}$ moreover has the form $\boldTheta_{n}(x_{1}, \dots, x_{n}, y)$, i.e.\ no parameters occur in~$\boldTheta_{n}$.
\end{definition}

  Note that in the definition of a parametrized local equational definition of $n$-filters we allow for $n \assign 0$, in which case the equivalence is interpreted as
\begin{align*}
  b \in \Fg^{\alg{A}}_{\logic{L}} \emptyset & \iff \alg{A} \vDash \boldTheta_{0}(b, \tuple{c}) \text{ for some } \boldTheta_{0} \in \boldPsi_{0} \text{ and some } \tuple{c} \in \alg{A}.
\end{align*}

\begin{definition}
  A \emph{(parametrized) local equational definition of compact $\logic{L}$-filters} on a class of algebras $\class{K}$ is a sequence $\boldPsi = (\boldPsi_{n})_{n \in \omega}$ of (parametrized) local equational definitions of $n$-generated $\logic{L}$-filters on $\class{K}$. If such a definition exists, we say that $\logic{L}$ \emph{has a (parametrized) local EDCF} on $\class{K}$, or more explicitly that it has the (parametrized) local EDCF($\boldPsi$) on $\class{K}$.
\end{definition}

\begin{definition}
 A \emph{(parametrized) global equational definition of compact $\logic{L}$-filters} on $\class{K}$ is a sequence $\boldTheta = (\boldTheta_{n})_{n \in \omega}$ of (parametrized) global equational definitions of $n$-generated $\logic{L}$-filters on $\class{K}$. If such a definition exists, we say that $\logic{L}$ \emph{has a (parametrized) global EDCF} on $\class{K}$, or more explicitly that it has the (parametrized) global EDCF($\boldTheta$) on $\class{K}$. Whenever convenient, we simply call $\boldTheta$ a \emph{(parametrized) equational definition of compact $\logic{L}$-filters}.
\end{definition}

  Allowing for infinitary conjunctions and disjunctions and quantification over infinitely many variables, we can rewrite the above definitions as follows: $\boldPsi$ is a parametrized local equational definition of compact $\logic{L}$-filters on $\class{K}$ if for each $\alg{A} \in \class{K}$, each $n \in \omega$, and all $a_{1}, \dots, a_{n}, b \in \alg{A}$
\begin{align*}
  b \in \Fg^{\alg{A}}_{\logic{L}}(a_{1}, \dots, a_{n}) \iff \alg{A} \vDash \exists \tuple{z} \; \bigvee_{\mathclap{\boldTheta_{n} \in \boldPsi_{n}}} \; \bigwedge \boldTheta_{n}(a_{1}, \dots, a_{n}, b, \tuple{z}).
\end{align*}
  It is a parametrized equational definition if
\begin{align*}
  b \in \Fg^{\alg{A}}_{\logic{L}}(a_{1}, \dots, a_{n}) \iff \alg{A} \vDash \exists \tuple{z}  \bigwedge \boldTheta_{n}(a_{1}, \dots, a_{n}, b, \tuple{z}),
\end{align*}
  a local equational definition if
\begin{align*}
  b \in \Fg^{\alg{A}}_{\logic{L}}(a_{1}, \dots, a_{n}) \iff \alg{A} \vDash \; \bigvee_{\mathclap{\boldTheta_{n} \in \boldPsi_{n}}} \; \bigwedge \boldTheta_{n}(a_{1}, \dots, a_{n}, b),
\end{align*}
  and a global equational definition if
\begin{align*}
  b \in \Fg^{\alg{A}}_{\logic{L}}(a_{1}, \dots, a_{n}) \iff \alg{A} \vDash \bigwedge \boldTheta_{n}(a_{1}, \dots, a_{n}, b),
\end{align*}
  Clearly if $\logic{L}$ has an EDCF (of any of the above forms) on $\class{K}$, then it has an EDCF (of the given form) on any subclass of $\class{K}$.

  Given a logic $\logic{L}$, the most natural choice of the class $\class{K}$ is of course the algebraic counterpart $\Alg \logic{L}$. Our main results, however, will be proved for an arbitrary prevariety $\class{K} \supseteq \Alg \logic{L}$. This approach has two (related) advantages. Firstly, the variety or quasivariety generated by $\Alg \logic{L}$ may be more convenient to work with than $\Alg \logic{L}$ itself. Secondly, this avoids the need to first compute the class $\Alg \logic{L}$ before deploying the results proved below. Finding a convenient class $\class{K} \supseteq \Alg \logic{L}$ is sometimes easier than precisely determining the class $\Alg \logic{L}$.

  Some logics fail to have an EDCF only for the trivial reason that they do not have a theorem. Consider, for instance, the fragment of local modal logic $\LocK$ in the signature $\Box, \wedge, \vee$. In the absence of the constant~$1$, the smallest filter of this fragment of $\LocK$ on each algebra is the empty filter. But the empty filter does not have a parametrized equational definition on the singleton algebra. Apart from this small deficiency, however, $\LocK$ enjoys most of the benefits of the EDCF. It would therefore be unfair to lump it together with logics where the EDCF fails for more substantial reasons. This motivates the following definition.

\begin{definition}
 If $\logic{L}$ has an equational definition (of any of the above forms) of $n$-generated filters for each $n \geq 1$  on $\class{K}$, we say that $\logic{L}$ \emph{almost has an EDCF} (of the given form) on~$\class{K}$.
\end{definition}

\begin{fact}
  $\logic{L}$ has a (parametrized) EDCF on a subdirect class $\class{K}$ if and only if it has a theorem and it almost has a (parametrized) EDCF on $\class{K}$.
\end{fact}

\begin{proof}
  We only consider the global case, since the parametrized case is entirely analogous. Left to right, if $\logic{L}$ has no theorems, then $\emptyset$ is an $\logic{L}$-filter on the singleton algebra $\{ \ast \}$. (This algebra belongs to every subdirect class by virtue of being the empty subdirect product.) But every equation is satisfied in $\{ \ast \}$ by the element $\ast$, so no set of equations (not even the empty set) can define $\Fg^{\alg{A}}_{\logic{L}} \emptyset$. Each logic with an EDCF therefore has a theorem.

  Conversely, if $\logic{L}$ has a theorem and almost has an EDCF, then in particular it has a theorem in at most one variable $\varphi(x)$ and it has an equational definition of $1$-generated $\logic{L}$-filters $\boldTheta_{1}$. Then
\begin{align*}
  b \in \Fg^{\alg{A}}_{\logic{L}} \emptyset \iff b \in \Fg^{\alg{A}}_{\logic{L}} \varphi(b) \iff \alg{A} \vDash \boldTheta_{1}(\varphi(b), b),
\end{align*}
  so $\logic{L}$ has an EDCF if we take $\boldTheta_{0}(x) \assign \boldTheta_{1}(\varphi(x), x)$.
\end{proof}

  In contrast, logics without any theorems can enjoy a local EDCF, since on every algebra the empty family $\emptyset$ is a local equational definition of the empty set. (Indeed $b \in \Fg^{\alg{A}}_{\logic{L}} \emptyset$ if and only if there is some $\boldTheta_{0} \in \emptyset$ such that $\alg{A} \vDash \boldTheta_{0}(b)$, since $\Fg^{\alg{A}}_{\logic{L}} \emptyset = \emptyset$.) To construct an example of a logic with a local EDCF but without any theorems, consider any logic $\logic{L}$ with the local EDCF($\boldPsi$). Like every logic, $\logic{L}$ has a theoremless variant $\logic{L}'$ such that $\Gamma \vdash_{\logic{L}'} \varphi$ if and only if $\Gamma$ is non-empty and $\Gamma \vdash_{\logic{L}} \varphi$. Then $\logic{L}'$ is a theoremless logic with the local EDCF: its filters are precisely the filters of $\logic{L}'$ plus the empty filter on every algebra, so replacing $\boldPsi_{0}$ with $\emptyset$ does the job.

  The results about the different forms of EDCF proved in this paper have obvious variants for the corresponding forms of almost EDCF, obtained by restricting to non-empty filters. We shall not state these explicitly, but the reader should be aware that for logics without theorems the property of interest is typically (always, in the parametrized case) almost EDCF rather than EDCF.

  The definition of an EDCF for $\class{K} \supseteq \Alg \logic{L}$ can be restated in terms of relative congruences on arbitary algebras (instead of the equality relation on algebras in $\class{K}$). Recall that $\theta^{\alg{A}}_{\class{K}}$ denotes the smallest $\class{K}$-congruence on $\alg{A}$.

\begin{fact} \label{fact: alternative definition of edcf}
  Let $\boldPsi = (\boldPsi_{n})_{n \in \omega}$ be a sequence of families of sets of equations of the form $\boldTheta_{n}(x_{1}, \dots, x_{n}, y, \tuple{z})$, where the variables $x_{1}, \dots, x_{n}, y, \tuple{z}$ are assumed to be distinct and $\tuple{z}$ is a possibly infinite tuple of variables. Then the following are equivalent for each logic $\logic{L}$ and each subdirect class $\class{K}$:
\begin{enumerate}[(i)]
\item For each $n \in \omega$, each algebra $\alg{A}$, and each $a_{1}, \dots, a_{n}, b \in \alg{A}$
\begin{align*}
  b \in \Fg^{\alg{A}}_{\logic{L}}(a_{1}, \dots, a_{n}) & \iff \boldTheta_{n}(a_{1}, \dots, a_{n}, b, \tuple{c}) \subseteq \theta^{\alg{A}}_{\class{K}} \\ & \phantom{\iff}\, \text{ for some } \boldTheta_{n} \in \boldPsi_{n} \text{ and some } \tuple{c} \in \alg{A}.
\end{align*}
\item For each $n \in \omega$, each algebra $\alg{A}$, and each $a_{1}, \dots, a_{n}, b \in \alg{A}$
\begin{align*}
  b \in \Fg^{\alg{A}}_{\logic{L}}(a_{1}, \dots, a_{n}) & \iff \alg{A} / \theta^{\alg{A}}_{\class{K}} \vDash \boldTheta_{n}(a_{1} / \theta^{\alg{A}}_{\class{K}}, \dots, a_{n} / \theta^{\alg{A}}_{\class{K}}, b / \theta^{\alg{A}}_{\class{K}}, \tuple{c} / \theta^{\alg{A}}_{\class{K}}) \\ & \phantom{\iff}\, \text{ for some } \boldTheta_{n} \in \boldPsi_{n} \text{ and some } \tuple{c} \in \alg{A}.
\end{align*}
\item $\class{K} \supseteq \Alg \logic{L}$ and $\logic{L}$ has the parametrized local EDCF($\boldPsi$) on $\class{K}$.
\end{enumerate}
\end{fact}

\begin{proof}
  The equivalence of the right-hand sides of (i) and (ii) is clear.

  (iii) $\Rightarrow$ (ii): if $\class{K} \supseteq \Alg \logic{L}$, then $\theta^{\alg{A}}_{\class{K}} \leq \theta^{\alg{A}}_{\Alg \logic{L}}$. But $\theta^{\alg{A}}_{\Alg \logic{L}}$ is compatible with each $\logic{L}$-filter on $\alg{A}$ and the congruences compatible with an $\logic{L}$-filter form a downset in $\Con \alg{A}$, so
\begin{align*}
  b \in \Fg^{\alg{A}}_{\logic{L}}(a_{1}, \dots, a_{n}) & \iff b / \theta^{\alg{A}}_{\class{K}} \in \Fg^{\alg{A} / \theta^{\alg{A}}_{\class{K}}}_{\logic{L}}(a_{1} / \theta^{\alg{A}}_{\class{K}}, \dots, a_{n} / \theta^{\alg{A}}_{\class{K}}) \\
  & \iff \alg{A} / \theta^{\alg{A}}_{\class{K}} \vDash \boldTheta_{n}(a_{1} / \theta^{\alg{A}}_{\class{K}}, \dots, a_{n} / \theta^{\alg{A}}_{\class{K}}, b / \theta^{\alg{A}}_{\class{K}}, \tuple{c} / \theta^{\alg{A}}_{\class{K}}) \\
  & \phantom{\iff}\, \text{ for some } \boldTheta_{n} \in \boldPsi_{n} \text{ and some } \tuple{c} / \theta^{\alg{A}}_{\class{K}} \in \alg{A} / \theta^{\alg{A}}_{\class{K}}.
\end{align*}

  (ii) $\Rightarrow$ (iii): $\theta^{\alg{A}}_{\class{K}}$ is the identity congruence on each $\alg{A} \in \class{K}$, so if $\class{K} \supseteq \Alg \logic{L}$, then (iii) is a special case of (ii). It remains to show that (ii) implies that $\Alg \logic{L} \subseteq \class{K}$. Thus consider $\alg{A} \in \Alg \logic{L}$. It suffices to show that $\theta^{\alg{A}}_{\class{K}}$ is compatible with every $\logic{L}$-filter on $\alg{A}$, since then $\theta^{\alg{A}}_{\class{K}} \leq \theta^{\alg{A}}_{\Alg \logic{L}} = \Delta_{\alg{A}}$, so indeed $\Alg \logic{L} \subseteq \class{K}$. Given $n \geq 1$, this holds for $n$-generated $\logic{L}$-filters on $\alg{A}$ because by (ii) each such filter is the homomorphic preimage of an $n$-generated $\logic{L}$-filter on $\alg{A} / \theta^{\alg{A}}_{\class{K}}$ with respect to the projection map $\pi\colon \alg{A} \to \alg{A} / \theta^{\alg{A}}_{\class{K}}$. The $\logic{L}$-filter $\Fg^{\alg{A}}_{\logic{L}} \emptyset$ is either empty or it is a $1$-generated $\logic{L}$-filter, in either case $\theta^{\alg{A}}_{\class{K}}$ is compatible with $\Fg^{\alg{A}}_{\logic{L}} \emptyset$. Finally, since $\logic{L}$ is finitary, each $\logic{L}$-filter is a directed union of compact $\logic{L}$-filters, therefore $\theta^{\alg{A}}_{\class{K}}$ is compatible with all $\logic{L}$-filters. But the equality relation is the largest congruence on $\alg{A} \in \Alg \logic{L}$ compatible with each $\logic{L}$-filter, so $\theta^{\alg{A}}_{\class{K}} = \Delta_{\alg{A}}$.
\end{proof}

  The following theorem shows that having a parametrized local EDCF is a trivial property of logics, which on its own does not provide any information about the logic in question. What may provide non-trivial information about a logic is knowing a particular parametrized local EDCF for the logic.

\begin{theorem} \label{thm: pledcf}
  Every logic has a parametrized local EDCF on the class of all algebras (in the signature of the logic).
\end{theorem}

\begin{proof}
  Consider a logic $\logic{L}$ and an algebra $\alg{A}$. For each $X \subseteq \alg{A}$ we define ${C(X) \subseteq \alg{A}}$ so that $b \in C(X)$ if and only if there is a finitary rule ${\gamma_{1}, \dots, \gamma_{m} \vdash \varphi}$ valid in $\logic{L}$ and a homomorphism $h\colon \Fm \to \alg{A}$ such that $h(\gamma_{i}) \in X$ for each $i \in \{ 1, \dots, m \}$ and ${h(\varphi) = b}$. The map $C$ is monotone: if $X \subseteq Y$, then $C(X) \subseteq C(Y)$. It is also increasing: $X \subseteq C(X)$. Finally, it is finitary: if $a \in C(X)$, then there is a finite subset $Y \subseteq X$ such that $a \in C(Y)$. We shall write $C(a_{1}, \dots, a_{n})$ for $C(\{ a_{1}, \dots, a_{n} \})$.

  A set $X \subseteq \alg{A}$ is an $\logic{L}$-filter if and only if $C(X) \subseteq X$. Let us take
\begin{align*}
  C_{0}(X) & \assign X, & C_{i+1}(X) & \assign C(C_{i}(X)) \text{ for } i \in \omega, & C_{\omega}(X) & \assign \bigcup_{i \in \omega} F_{i}.
\end{align*}
  We know that $C_{\omega}(X) = \Fg^{\alg{A}}_{\logic{L}}(X)$~\cite[Theorem~2.23]{font16}. A brief self-contained proof of this fact runs as follows: if $X \subseteq F \in \Fi_{\logic{L}} \alg{A}$, then $C(X) \subseteq C(F) \subseteq F$, and thus $C_{\omega}(X) \subseteq F$. Conversely, if $a \in C(C_{\omega}(X))$, then $a \in C(C_{n}(X))$ for some $n \in \omega$ (because $C$ is finitary), and thus $a \in C_{n+1}(X) \subseteq C_{\omega}(X)$. Therefore $C(C_{\omega}(X)) \subseteq C_{\omega}(X)$ and $C_{\omega}(X)$ is an $\logic{L}$-filter.

  To prove that $\logic{L}$ has a parametrized local EDCF, observe first that there is a family of sets of equations $\boldPsi_{n}(x_{1}, \dots, x_{n}, y, \tuple{z})$, where $\tuple{z}$ is a tuple of variables $z_{i}$ for $i \in \omega$, such that
\begin{align*}
  b \in C(a_{1}, \dots, a_{n}) \iff \alg{A} \vDash \boldTheta_{n}(a_{1}, \dots, a_{n}, b, \tuple{c}) \text{ for some $\boldTheta_{n} \in \boldPsi_{n}$ and $\tuple{c} \in \alg{A}$.}
\end{align*}
  This family contains a set of equations $\boldTheta_{n}$ for each pair consisting of a finitary rule $\gamma_{1}, \dots, \gamma_{m} \vdash \varphi$ in the variables $\tuple{z}$ and a function $f\colon \{ 1, \dots, m \} \to \{ 1, \dots, n \}$, namely
\begin{align*}
  \boldTheta_{n}(x_{1}, \dots, x_{n}, y, \tuple{z}) \assign \{ \gamma_{1} \equals x_{f(1)}, \dots, \gamma_{n} \equals x_{f(n)}, \varphi \equals y \}.
\end{align*}

  Now $b \in \Fg^{\alg{A}}_{\logic{L}}(a_{1}, \dots, a_{n})$ if and only if $b \in C_{\omega}(a_{1}, \dots, a_{n})$, i.e.\ if and only if $b \in C_{m}(a_{1}, \dots, a_{n})$ for some $m \in \omega$. Because the map $C$ is finitary, this holds if and only if there are $d_{1}, \dots, d_{k} \in \alg{A}$ such that
\begin{align*}
  & d_{1} \in C(a_{1}, \dots, a_{n}), \\
  & d_{2} \in C(a_{1}, \dots, a_{n}, d_{1}), \\
  & \dots, \\
  & d_{k} \in C(a_{1}, \dots, a_{n}, d_{1}, \dots, d_{k-1}), \\
  & d_{k} = b.
\end{align*}
  By the equivalence proved in the previous paragraph, this holds if and only if there are tuples of parameters $\tuple{c}_{1}, \dots, \tuple{c}_{k} \in \alg{A}$, elements $d_{1}, \dots, d_{k} \in \alg{A}$, and sets of equations $\boldTheta_{n+i} \in \boldPsi_{n+i}$ for $i \in \{ 0, \dots, k - 1 \}$ such that for all $i \in \{ 0, \dots, k - 1 \}$ we have $\alg{A} \vDash \boldTheta_{n+i}(a_{1}, \dots, a_{n}, d_{i+1}, d_{1}, \dots, d_{i}, \tuple{c})$ and moreover $d_{k} = b$. But this condition can be rephrased as a parametrized local EDCF.
\end{proof}

  If a logic $\logic{L}$ has a parametrized local EDCF with respect to two families, $\boldPsi$ and $\boldPsi'$, then these families are in a sense equivalent.

\begin{fact} \label{fact: edcf unique}
  Suppose that $\logic{L}$ has the parametrized local EDCF($\boldPsi$) on a subdirect class $\class{K}$. Then $\logic{L}$ also has the parametrized local EDCF($\boldPsi'$) on $\class{K}$ if and only if for each $\boldTheta_{n} \in \boldPsi_{n}$ there is some $\boldTheta'_{n} \in \boldPsi'_{n}$ such that
\begin{align*}
  \class{K} \vDash \exists \tuple{z} \bigwedge \boldTheta_{n}(x_{1}, \dots, x_{n}, y, \tuple{z}) \implies \exists \tuple{z} \bigwedge \boldTheta'_{n}(x_{1}, \dots, x_{n}, y, \tuple{z})
\end{align*}
  and for each $\boldTheta'_{n} \in \boldPsi'_{n}$ there is some $\boldTheta_{n} \in \boldPsi_{n}$ such that
\begin{align*}
  \class{K} \vDash \exists \tuple{z} \bigwedge \boldTheta'_{n}(x_{1}, \dots, x_{n}, y, \tuple{z}) \implies \exists \tuple{z} \bigwedge \boldTheta_{n}(x_{1}, \dots, x_{n}, y, \tuple{z}).
\end{align*}
\end{fact}

\begin{proof}
  The verification of the right-to-left direction is straightforward and left to the reader. Conversely, suppose that both $\boldPsi$ and $\boldPsi'$ are parametrized local EDCFs for $\logic{L}$ on $\class{K}$. Then
\begin{align*}
  \class{K} \vDash \exists \tuple{z} \bigwedge \boldTheta_{n}(x_{1}, \dots, x_{n}, y, \tuple{z}) \implies \bigvee_{\boldTheta'_{n} \in \boldPsi'_{n}} \left( \exists \tuple{z} \bigwedge \boldTheta'_{n}(x_{1}, \dots, x_{n}, y, \tuple{z}) \right).
\end{align*}
  We prove the first of the two asserted implications by contradiction. The second implication will then follow by symmetry. If for each $\boldTheta'_{n} \in \boldPsi'_{n}$
\begin{align*}
  \class{K} \nvDash \exists \tuple{z} \bigwedge \boldTheta_{n}(x_{1}, \dots, x_{n}, y, \tuple{z}) \implies \exists \tuple{z} \bigwedge \boldTheta'_{n}(x_{1}, \dots, x_{n}, y, \tuple{z}),
\end{align*}
  as witnessed by a valuation $v_{\boldTheta'_{n}}$ on $\alg{A}_{\boldTheta'_{n}} \in \class{K}$, then the algebra $\alg{A} \assign \prod_{\boldTheta'_{n} \in \boldPsi'_{n}} \alg{A}_{\boldTheta'_{n}}$ and the product of the valuations $v_{\boldTheta'_{n}}$ would witness that
\begin{align*}
  \class{K} \nvDash \exists \tuple{z} \bigwedge \boldTheta_{n}(x_{1}, \dots, x_{n}, y, \tuple{z}) \implies \bigvee_{\boldTheta'_{n} \in \boldPsi'_{n}} \left( \exists \tuple{z} \bigwedge \boldTheta'_{n}(x_{1}, \dots, x_{n}, y, \tuple{z}) \right),
\end{align*}
  contradicting the fact that $\boldPsi$ and $\boldPsi'$ are parametrized local EDCFs.
\end{proof}

\begin{theorem} \label{thm: removing locality}
  Suppose that $\logic{L}$ has the (parametrized) local EDCF($\boldPsi$) on $\class{K}$. If $\logic{L}$ in addition has a (parametrized) EDCF, then it has a (parametrized) EDCF($\boldTheta$) on $\class{K}$ for some $\boldTheta$ such that $\boldTheta_{n} \in \boldPsi_{n}$ for each $n \in \omega$.
\end{theorem}

\begin{proof}
  Suppose that $\logic{L}$ has a parametrized local EDCF($\boldPsi$) and a parametrized EDCF($\boldTheta'$) on $\class{K}$. The previous fact ensures that there is some $\boldTheta \in \boldPsi$ such that
\begin{align*}
  \class{K} \vDash \exists \tuple{z} \bigwedge \boldTheta'_{n}(x_{1}, \dots, x_{n}, y, \tuple{z}) \implies \exists \tuple{z} \bigwedge \boldTheta_{n}(x_{1}, \dots, x_{n}, y, \tuple{z}).
\end{align*}
  If $b \in \Fg^{\alg{A}}_{\logic{L}}(a_{1}, \dots, a_{n})$ for $a_{1}, \dots, a_{n}, b \in \alg{A} \in \class{K}$, then the parametrized EDCF($\boldTheta$) yields that $\alg{A} \vDash \exists \tuple{z} \bigwedge \boldTheta'_{n}(a_{1}, \dots, a_{n}, b, \tuple{z})$, so $\alg{A} \vDash \exists \tuple{z} \bigwedge \boldTheta_{n}(a_{1}, \dots, a_{n}, b, \tuple{z})$ by the above implication. Conversely, $\alg{A} \vDash \exists \tuple{z} \bigwedge \boldTheta_{n}(a_{1}, \dots, a_{n}, b, \tuple{z})$ implies that $b \in \Fg^{\alg{A}}_{\logic{L}}(a_{1}, \dots, a_{n})$ by the parametrized local EDCF($\boldPsi$). The logic $\logic{L}$ therefore has the parametrized EDCF($\boldTheta$) on $\class{K}$. Clearly if $\boldPsi$ does not involve any parameters, then neither does $\boldTheta$.
\end{proof}

\begin{corollary} \label{cor: global edcf}
  A logic has an EDCF on a class of algebras $\class{K}$ if and only if it has both a local EDCF and a parametrized EDCF on $\class{K}$.
\end{corollary}

  We now move on to parametrized EDCFs, where the general sequence $\boldPsi$ of families of sets of formulas $\boldPsi_{n}(x_{1}, \dots, x_{n}, y, \tuple{z})$ is replaced by a sequence $\boldTheta$ of sets of formulas $\boldTheta_{n}(x_{1}, \dots, x_{n}, y, \tuple{z})$. We now show under what conditions a parametrized local EDCF can be pared down to a parametrized EDCF.

\begin{definition}
  A logic $\logic{L}$ is said to have \emph{absolute factor determined compact filters} on a class of algebras $\class{K}$ if for each family of algebras $\alg{A}_{i} \in \class{K}$ with $i \in I$ and each ${a_{1}, \dots, a_{n} \in \alg{A}} \assign \prod_{i\in I} \alg{A}_{i}$ with $n \in \omega$
\begin{align*}
  \Fg^{\alg{A}}_{\logic{L}} (a_{1}, \dots, a_{n}) = \prod_{i \in I} \Fg^{\alg{A}_{i}}_{\logic{L}}(\pi_{i}(a_{1}), \dots, \pi_{i}(a_{n})),
\end{align*}
  where $\pi_{i}\colon \alg{A} \to \alg{A}_{i}$ are the projection maps.
\end{definition}

\begin{definition}
  A \emph{$n$-test algebra for compact filters} of $\logic{L}$ on a class $\class{K}$ for $n \in \omega$ is an algebra $\alg{A}_{n} \in \class{K}$ with elements $p_{1}, \dots, p_{n}, q \in \alg{A}_{n}$ (called \emph{test elements}) such that for each $\alg{A} \in \class{K}$
\begin{align*}
  b \in \Fg^{\alg{A}}_{\logic{L}}(a_{1}, \dots, a_{n}) & \iff \text{there is a homomorphism } h\colon \alg{A}_{n} \to \alg{A} \text{ with } \\
  & \phantom{\iff} ~ h(p_{1}) = a_{1}, \dots, h(p_{n}) = a_{n}, \text{ and } h(q) = b.
\end{align*}
  Equivalently, $q \in \Fg^{\alg{A}_{n}}_{\logic{L}}(p_{1}, \dots, p_{n})$ and moreover $b \in \Fg^{\alg{A}}_{\logic{L}}(a_{1}, \dots, a_{n})$ implies the existence of the above homomorphism $h$. A logic $\logic{L}$ \emph{has test algebras for compact filters} on $\class{K}$ if it has an $n$-test algebra on $\class{K}$ for all~$n \in \omega$.
\end{definition}

\begin{theorem} \label{thm: pedcf} \label{thm: pedfc equiv}
  Let $\logic{L}$ be a logic which has the parametrized local EDCF($\boldPsi$) on a subdirect class $\class{K}$. Then the following are equivalent:
\begin{enumerate}[(i)]
\item $\logic{L}$ has the parametrized EDCF($\boldTheta$) on $\class{K}$ for some $\boldTheta$ with $\boldTheta_{n} \in \boldPsi_{n}$.
\item $\logic{L}$ has a parametrized EDCF on $\class{K}$.
\item $\logic{L}$ has absolute factor determined compact filters on $\class{K}$.
\item $\logic{L}$ has test algebras for compact filters with respect to $\class{K}$.
\end{enumerate}
\end{theorem}

\begin{proof}
  (i)$\Rightarrow$(ii): trivial. (ii)$\Rightarrow$(iii): suppose that $\logic{L}$ has the parametrized EDCF($\boldTheta$) on $\class{K}$. Consider a family of algebras $\alg{A}_{i} \in \class{K}$ for $i \in I$. Take $\alg{A} \assign \prod_{i \in I} \alg{A}_{i}$ and let $\pi_{i}\colon \alg{A} \to \alg{A}_{i}$ be the projection maps. For all $a_{1}, \dots, a_{n} \in \alg{A}$ the inclusion $\Fg^{\alg{A}}(a_{1}, \dots, a_{n}) \subseteq \prod_{i \in I} \Fg^{\alg{A}_{i}}_{\logic{L}}(\pi_{i}(a_{1}), \dots, \pi_{i}(a_{n}))$ holds for any logic. Conversely, take $b \in \prod_{i \in I} \Fg^{\alg{A}_{i}}_{\logic{L}} (\pi_{i}(a_{1}), \dots, \pi_{i}(a_{n}))$. That is, $\pi_{i}(b) \in \Fg^{\alg{A}_{i}}_{\logic{L}}(\pi_{i}(a_{1}), \dots, \pi_{i}(a_{n})$ for each $i \in I$. The parametrized EDCF($\boldTheta$) now yields a tuple $\tuple{c}_{i} \in \alg{A}_{i}$ for each $i \in I$ such that ${\alg{A}_{i} \vDash \boldTheta_{n}(\pi_{i}(a_{1}), \dots, \pi_{i}(a_{n}), \pi_{i}(b), \tuple{c}_{i})}$ for each $i \in I$. Taking $\tuple{c} \in \alg{A}$ to be the unique tuple such that $\pi_{i}(\tuple{c}) = \tuple{c}_{i}$ for each $i \in I$, we obtain that $\alg{A}_{i} \vDash \boldTheta_{n}(\pi_{i}(a_{1}), \dots, \pi_{i}(a_{n}), \pi_{i}(b), \pi_{i}(\tuple{c}))$ for each $i \in I$, so $\alg{A} \vDash \boldTheta_{n}(a_{1}, \dots, a_{n}, b, \tuple{c})$. The parametrized EDCF($\boldTheta$) now yields that $b \in \Fg^{\alg{A}}_{\logic{L}}(a_{1}, \dots, a_{n})$.

  (iii)$\Rightarrow$(iv): the proof of this implication is essentially identical to the proof of \cite[Theorem~2.4.1, (ii)$\Rightarrow$(viii)]{czelakowski01}. Let $\kappa$ be the $\aleph_{0}$ plus cardinality of the tuple of parameters $\tuple{z}$ associated with the parametrized local EDCF~$\boldPsi$. Consider all tuples $\langle \alg{A}, a_{1}, \dots, a_{n}, b \rangle$ where $\alg{A} \in \class{K}$ is an algebra with $a_{1}, \dots, a_{n}, b \in \alg{A}$ such that $b \in \Fg^{\alg{A}}_{\logic{L}} (a_{1}, \dots, a_{n})$ and $\alg{A}$ is generated by a set of cardinality at most $\kappa$. Up to isomorphism, there is only a set of such structures $\mathbb{A} \assign \langle \alg{A}, a_{1}, \dots, a_{n}, b \rangle$. Let $S$ be a set which contains at least one representative of each isomorphism class. Take $\alg{A}_{n} \assign \prod_{\mathbb{A} \in S} \alg{A}$, and let $\pi_{\mathbb{A}}\colon \alg{A}_{n} \to \alg{A}$ be the projection maps. Then take $p_{i}, q \in \alg{A}_{n}$ such that $\pi_{\mathbb{A}}(p_{1}) = a_{1}, \dots, \pi_{\mathbb{A}}(p_{n}) = a_{n}, \pi_{\mathbb{A}}(q) = b$ for each structure $\mathbb{A} = \langle \alg{A}, a_{1}, \dots, a_{n}, b \rangle \in S$.

  We claim that $\alg{A}_{n}$ is an $n$-test algebra whose test elements are $p_{1}, \dots, p_{n}, q$. By construction $\pi_{\mathbb{A}}(q) \in \Fg^{\alg{A}}_{\logic{L}} (\pi_{\mathbb{A}}(p_{1}), \dots, \pi_{\mathbb{A}}(p_{n}))$, so by (iii) we can infer that $q \in \Fg^{\alg{A}_{n}}_{\logic{L}} (p_{1}, \dots, p_{n})$. If $b \in \Fg^{\alg{A}}_{\logic{L}} (a_{1}, \dots, a_{n})$, then by the parametrized \mbox{local} EDCF($\boldPsi$) there is a subalgebra $\alg{B}$ of $\alg{A}$ generated by a set of cardinality at most $\kappa$ such that $a_{1}, \dots, a_{n}, b \in \alg{B}$ and $b \in \Fg^{\alg{B}}_{\logic{L}} (a_{1}, \dots, a_{n})$. By the definition of $S$ we have ${\langle \alg{B}, a_{1}, \dots, a_{n}, b \rangle \in S}$, so the projection map is the required homomorphism $h\colon \alg{A}_{n} \to \alg{B}$ with $h(p_{1}) = a_{1}$, \dots, $h(p_{n}) = a_{n}$, $h(q) = b$.

  (iv)$\Rightarrow$(i): Suppose that for each $n \in \omega$ there is an $n$-test algebra $\alg{A}_{n}$ for $\logic{L}$ with respect to $\class{K}$, with test elements $p_{1}, \dots, p_{n}, q \in \alg{A}_{n}$. Since $q \in \Fg^{\alg{A}_{n}}_{\logic{L}}(p_{1}, \dots, p_{n})$, we obtain that $\alg{A}_{n} \vDash \boldTheta_{n}(p_{1}, \dots, p_{n}, q, \tuple{r})$ for some $\boldTheta_{n} \in \boldPsi_{n}$ and some tuple $\tuple{r} \in \alg{A}_{n}$. We claim that $\logic{L}$ has the parametrized EDCF($\boldTheta$). One direction holds trivially: if $\alg{A} \vDash \boldTheta_{n}(a_{1}, \dots, a_{n}, b, \tuple{c})$ with $\alg{A} \in \class{K}$ for some tuple $\tuple{c} \in \alg{A}$, then $b \in \Fg^{\alg{A}}_{\logic{L}}(a_{1}, \dots, a_{n})$ by the parametrized local EDCF($\boldPsi$). Conversely, suppose that $b \in \Fg^{\alg{A}}_{\logic{L}}(a_{1}, \dots, a_{n})$. Then there is a homomorphism $h\colon \alg{A}_{n} \to \alg{A}$ such that $h(p_{1}) = a_{1}, \dots, h(p_{n}) = a_{n}, h(q) = b$. But $\alg{A}_{n} \vDash \boldTheta_{n}(p_{1}, \dots, p_{n}, q, \tuple{r})$ for some $\tuple{r} \in \alg{A}_{n}$, so $\alg{A} \vDash \boldTheta_{n}(h(p_{1}), \dots, h(p_{n}), h(q), h(\tuple{r}))$ and indeed $\alg{A} \vDash \boldTheta_{n}(a_{1}, \dots, a_{n}, b, \tuple{c})$ for some tuple $\tuple{c} \in \alg{A}$.

  (The equivalence (i) $\Leftrightarrow$ (ii) was of course already proved in Theorem~\ref{thm: removing locality}. However, going through the implication (iv) $\Rightarrow$ (i) appears to be the most natural way to prove the implication (iv) $\Rightarrow$ (ii).)
\end{proof}

  The following is a less refined but more concise form of the above theorem.

\begin{theorem} \label{thm: baby pedcf}
  The following are equivalent for every logic $\logic{L}$:
\begin{enumerate}[(i)]
\item $\logic{L}$ has a parametrized EDCF on $\Alg \logic{L}$.
\item $\logic{L}$ has absolute factor determined compact filters on $\Alg \logic{L}$.
\item $\logic{L}$ has test algebras for compact filters with respect to $\Alg \logic{L}$.
\end{enumerate}
\end{theorem}

\begin{proof}
  By Theorem~\ref{thm: pledcf} each logic $\logic{L}$ has a parametrized local EDCF on the class of all algebras. The equivalence then follows from the previous theorem.
\end{proof}

  By way of illustration, let us use the above theorem to show that the local modal logic $\LocK$ and the paraconsistent weak Kleene logic $\PWK$ do not have a parametrized EDCF on their algebraic counterparts.

  We have already seen that $\LocK$ has a local EDCF on its algebraic counterpart, which is the variety $\BAO$ of Boolean algebras with an operator.

\begin{example}
  $\LocK$ does not have a parametrized EDCF on $\Alg \LocK = \BAO$.
\end{example}

\begin{proof}
  The filters of $\LocK$ on such algebras are the non-empty lattice filters closed under $\Box$, so for $a_{1}, \dots, a_{n}, b \in \alg{A} \in \BAO$
\begin{align*}
  b \in \Fg^{\alg{A}}_{\LocK}(a_{1}, \dots, a_{n}) \iff \alg{A} \vDash \Box_{k} (a_{1} \wedge \dots \wedge a_{n}) \leq b \text{ for some } k \in \omega,
\end{align*}
  where $\Box_{0} x \assign x$ and $\Box_{i+1} x \assign x \wedge \Box \Box_{i} x$. For each $n$ there is some $\alg{A}_{n} \in \BAO$ and some $a_{n} \in \alg{A}_{n}$ such that $\Box_{n} a_{n} < \Box_{n-1} a_{n}$. Take $\alg{A} \assign \prod_{n \in \omega} \alg{A}_{n}$ and $a \assign \prod_{n \in \omega} a_{n}$. Let $b \in \alg{A}$ be the tuple such that $b_{n} = \Box_{n} a_{n}$ for each $n \in \omega$. Then $b \in \prod_{n \in \omega} \Fg^{\alg{A}_{n}}_{\LocK} b_{n}$, but $b \notin \Fg^{\alg{A}}_{\LocK} a$: otherwise there would be some $k \in \omega$ such that $\Box_{k} a \leq b$, but then $\Box_{k} a_{k+1} \leq b_{k+1} = \Box_{k+1} a_{k+1}$, contradicting the assumption that $\Box_{k+1} a_{k+1} < \Box_{k} a_{k+1}$.

  (Alternatively, one may apply Theorem~\ref{thm: removing locality}: if $\LocK$ had a parametrized EDCF, it would have a global EDCF $\boldTheta$ where $\boldTheta_{1}(x_{1}, y) \assign \{ \Box_{k} x_{1} \leq y \}$ for some $k \in \omega$. But this is directly contradicted by the fact that $\Box_{k+1} a_{k+1} \in \Fg^{\alg{A}_{k+1}}_{\LocK} (a_{k+1})$ but $\Box_{k} a_{k+1} \nleq \Box_{k+1} a_{k+1}$.)
\end{proof}

  Paraconsistent weak Kleene logic $\PWK$, studied in detail in~\cite[Chapter~7]{bonzio+paoli+prabaldi22}, is the logic determined by the three-element matrix $\langle \WKthree, \{ 1, \onehalf \} \rangle$, where $\WKthree \assign \langle \{ 0, 1, \onehalf \}, \vee, \neg \rangle$ with
\begin{align*}
  & \neg 0 \assign 1, & & \neg 1 \assign 0, & & \neg \onehalf \assign \onehalf,
\end{align*}
  and with $\vee$ defined as the binary join operation in the chain $0 < 1 < \onehalf$. Sometimes the operation $\wedge$ defined as the binary meet operation in the (different!) three-element chain $\onehalf < 0 < 1$ is added to the signature of $\WKthree$, but this operation is definable as $x \wedge y \assign \neg (\neg x \vee \neg y)$. The values $1$ and $0$ correspond to the true and false values of classical logic, while the value $\onehalf$ is infectious in the sense that if any subformula of a formula is evaluated to $\onehalf$, then so is the entire formula.

  The variety generated by $\WKthree$ is called the variety of \emph{generalized involutive bisemilattices} and denoted by $\GIB$~\cite[Proposition~2.4.22]{bonzio+paoli+prabaldi22}. The algebraic counterpart of $\PWK$ is the quasivariety generated by $\WKthree$~\cite[Corollary~7.1.16 and Theorem~7.1.19]{bonzio+paoli+prabaldi22}, which is axiomatized relative to $\GIB$ by the quasi-equation
\begin{align*}
  x \equals \neg x ~ \& ~ y \equals \neg y \implies x \equals y.
\end{align*}

  Part of the interest of the logic $\PWK$ stems from the fact that it equivalently arises as the so-called left variable inclusion companion of classical logic~$\CL$:
\begin{align*}
  \Gamma \vdash_{\PWK} \varphi \iff \Delta \vdash_{\CL} \varphi \text{ for some } \Delta \subseteq \Gamma \text{ such that } \Var(\Delta) \subseteq \Var(\varphi),
\end{align*}
  where $\Var(\Delta)$ and $\Var(\varphi)$ are the sets of variables which occur in the set $\Delta$ and in the formula $\varphi$. One can infer from this that $\PWK$ is not protoalgebraic.

\begin{example} \label{example: pwk does not have pedcf}
  $\PWK$ does not have a parametrized EDCF on $\Alg \PWK$.
\end{example}

\begin{proof}
  Consider the map $h\colon \WKthree \times \WKthree \to \WKthree$ defined as follows:
\begin{align*}
  h(\pair{a}{b}) \assign
  \begin{cases}
    \onehalf & \text{ if either } a = \onehalf \text{ or } b = \onehalf, \\
    1 & \text{ if } \pair{a}{b} = \pair{1}{1} \text{ or } \pair{a}{b} = \pair{0}{1}, \\
    0 & \text{ if } \pair{a}{b} = \pair{0}{0} \text{ or } \pair{a}{b} = \pair{1}{0}.
  \end{cases}
\end{align*}
  This map is easily seen to be a homomorphism. Then the set $F \assign h^{-1}[\{ 1, \onehalf \}]$ is a filter of $\PWK$ on $\WKthree \times \WKthree$, due to being a homomorphic preimage of a filter of $\PWK$ on $\WKthree$. But $\pair{\onehalf}{0} \in F$ and $\pair{1}{0} \notin F$, so $\pair{1}{0} \notin \Fg^{\WKthree \times \WKthree}_{\PWK} (\pair{\onehalf}{0})$. On the other hand, $1 \in \Fg^{\WKthree}(\onehalf)$ and $0 \in \Fg^{\WKthree}(0)$, so $\pair{1}{0} \in \Fg^{\WKthree}(\onehalf) \times \Fg^{\WKthree}(0)$. This shows that $\Fg^{\WKthree \times \WKthree}_{\PWK} (\pair{\onehalf}{0}) \neq \Fg^{\WKthree}(\onehalf) \times \Fg^{\WKthree}(0)$, hence $\PWK$ does not have factor determined compact filters. Consequently, by Theorem~\ref{thm: baby pedcf} the logic $\PWK$ does not have a parametrized EDCF on $\Alg \logic{L}$.
\end{proof}

  The parametrized EDCF arises naturally when one studies the definability of ideals in algebras. Indeed, a variant of this notion, namely the parametrized equational definability of principal ideals, was already investigated by Aglian\`{o} and Ursini~\cite{agliano+ursini97} in the more restricted setting of subtractive varieties. Theorem~\ref{thm: baby pedcf} echoes~Proposition~2.5 and Theorem~2.6 of~\cite{agliano+ursini97}.

  The above characterization of the parametrized EDCF is, unsurprisingly, reminiscent of the known characterization of the parametrized DDT (Theorem~2.4.1 of Czelakowski~\cite{czelakowski01}). In particular, the properties of having absolute factor determined compact filters and test algebras are similar to but distinct from the properties of having factor determined compact filters and test matrices considered in~\cite{czelakowski01,agliano+ursini97}.\footnote{To be more precise about the terminology, Czelakowski~\cite[p.~140]{czelakowski01} talks about factor determined finitely generated filters and property (F), respectively, while Agliano and Ursini~\cite[p.~365]{agliano+ursini97} talk about test algebras for principal ideals and factorable principal ideals on direct products.} We shall discuss the relationship between these properties in more detail in Section~\ref{sec: edcf and ddt}.

  The parametrized EDCF on its own (without the global EDCF) is not a condition typical of the best-known examples on non-classical logics. Since, as we shall see (Theorem~\ref{thm: pddt = pedcf}), the parametrized EDCF and the parametrized DDT coincide for weakly algebraizable logics, we can recycle the logic of ideals of commutative unital rings, which served as an example of a logic with the parametrized (but not the global) DDT in~\cite[p.~143]{czelakowski01}, as an example of a logic with the parametrized (but not the global) EDCF.

\section{Global and local EDCF}
\label{sec: ledcf and edcf}

  In this section, we characterize which logics admit a local and a global EDCF without parameters. Given that every logic has a parametrized local EDCF and parametrized global EDCFs are not typical of the best-known families of non-classical logics, these two characterizations can be counted as the central results of the present paper.

  Recall that $\theta^{\alg{A}}_{\class{K}}$ denotes the smallest $\class{K}$-congruence on $\alg{A}$. Observe that since $\Alg \logic{L}$ is closed under isomorphic images and subdirect products, the prevariety generated by $\Alg \logic{L}$ is simply $\AlgS(\Alg \logic{L})$.

\begin{definition}
  A logic $\logic{L}$ enjoys the \emph{absolute filter extension property} on a prevariety $\class{K}$, or the \emph{absolute FEP on $\class{K}$} for short, if for all algebras $\alg{A} \leq \alg{B} \in \class{K}$ every $\logic{L}$-filter $F$ on $\alg{A}$ is the restriction to $\alg{A}$ of some $\logic{L}$-filter $G$ on $\alg{B}$, i.e.\ $F = G \cap \alg{A}$, or equivalently if for all algebras $\alg{A} \leq \alg{B} \in \class{K}$ every compact $\logic{L}$-filter $F$ on $\alg{A}$ is the restriction to $\alg{A}$ of some compact $\logic{L}$-filter $G$ on $\alg{B}$.
\end{definition}

  If $F$ is the restriction of some filter on $G$ to $\alg{A}$, then in particular it is the restriction of $\Fg^{\alg{B}}_{\logic{L}} F$ to $\alg{A}$, which accounts for the equivalence of the two conditions in the above definition, and also for the equivalence of the first two condition in the following fact.

\begin{fact}
  The following are equivalent for each prevariety $\class{K} \supseteq \AlgS(\Alg \logic{L})$:
\begin{enumerate}[(i)]
\item $\logic{L}$ enjoys the absolute FEP on $\class{K}$.
\item If $\alg{A} \leq \alg{B} \in \class{K}$, then $\Fg^{\alg{A}}_{\logic{L}} (a_{1}, \dots, a_{n}) = \alg{A} \cap \Fg^{\alg{B}}_{\logic{L}} (a_{1}, \dots, a_{n})$.
\item If $\alg{A} \leq \alg{B}$ and $\theta^{\alg{A}}_{\class{K}} = \alg{A}^{2} \cap \theta^{\alg{B}}_{\class{K}}$, then $\Fg^{\alg{A}}_{\logic{L}} (a_{1}, \dots, a_{n}) = \alg{A} \cap \Fg^{\alg{B}}_{\logic{L}} (a_{1}, \dots, a_{n})$.
\end{enumerate}
\end{fact}

\begin{proof}
  We only prove the equivalence of (ii) and (iii). Clearly (ii) is a special case of (iii). Conversely, suppose that (ii) holds and consider $\alg{A} \leq \alg{B}$ with $\theta^{\alg{A}}_{\class{K}} = \alg{A}^{2} \cap \theta^{\alg{B}}_{\class{K}}$. Then $\alg{A} / \theta^{\alg{A}}_{\class{K}} \leq \alg{B} / \theta^{\alg{B}}_{\class{K}} \in \class{K}$, so for each~$b \in \alg{A}$
\begin{align*}
  b \in \Fg^{\alg{A}}_{\logic{L}} (a_{1}, \dots, a_{n}) & \iff b / \theta^{\alg{A}}_{\class{K}} \in \Fg^{\alg{A} / \theta^{\alg{A}}_{\class{K}}} (a_{1} / \theta^{\alg{A}}_{\class{K}}, \dots, a_{n} / \theta^{\alg{A}}_{\class{K}}) \\
  & \iff b / \theta^{\alg{B}}_{\class{K}} \in \Fg^{\alg{B} / \theta^{\alg{B}}_{\class{K}}}(a_{1} / \theta^{\alg{B}}_{\class{K}}, \dots, a_{n} / \theta^{\alg{B}}_{\class{K}}) \\
  & \iff b \in \Fg^{\alg{B}} (a_{1}, \dots, a_{n}). \qedhere
\end{align*}
\end{proof}

  We now show that the local EDCF is, under the modest assumption that $\Alg \logic{L}$ is closed under subalgebras, equivalent to the absolute FEP. The implication from the local EDCF to the absolute FEP will be straight\-forward to prove. The opposite implication requires us to produce a family $\boldPsi_{n}$ of sets of equations in $n+1$ variables for each $n \in \omega$. To accomplish this, it will be convenient to introduce the notion of filter generation relative to a congruence.

  Recall that a congruence $\theta$ on $\alg{A}$ is compatible with a set $F \subseteq \alg{A}$ in case $a \in F$ and $\pair{a}{b} \in \theta$ imply $b \in F$. We now view at this relation from the opposite perspective: we say that a filter $F$ is compatible with $\theta$ if this implication holds. The $\logic{L}$-filters $F$ on $\alg{A}$ compatible with a given congruence $\theta$ form a complete lattice $\Fi^{\theta}_{\logic{L}} \alg{A}$ ordered by inclusion where meets are intersections. Then
\begin{align*}
  \Fi^{\theta}_{\logic{L}} \alg{A} \iso \Fi_{\logic{L}} \alg{A} / \theta \text{ via the map } \pi^{-1}\colon \Fi_{\logic{L}} \alg{A} / \theta \to \Fi^{\theta}_{\logic{L}} \alg{A},
\end{align*}
  where $\pi\colon \alg{A} \to \alg{A} / \theta$ is the projection map. However, we shall think about $\Fi^{\theta}_{\logic{L}} \alg{A}$ as an object in its own right sitting inside $\Fi_{\logic{L}} \alg{A}$. In particular, the smallest $\logic{L}$-filter compatible with $\theta$ which extends a set $X \subseteq \alg{A}$ is
\begin{align*}
  \Fg^{\alg{A}, \theta}_{\logic{L}} X \assign \pi^{-1}[\Fg^{\alg{A} / \theta}_{\logic{L}} \pi[X]].
\end{align*}
  In particular,
\begin{align*}
  b \in \Fg^{\alg{A}, \theta}_{\logic{L}} (a_{1}, \dots, a_{n}) \iff b / \theta \in \Fg^{\alg{A} / \theta} (a_{1} / \theta, \dots, a_{n} / \theta).
\end{align*}
  Observe that the map
\begin{align*}
  \Fg^{\alg{A}, \theta}_{\logic{L}}\colon \Fi_{\logic{L}} \alg{A} \to \Fi^{\theta}_{\logic{L}} \alg{A}
\end{align*}
  preserves arbitrary joins and maps compact elements of $\Fi_{\logic{L}} \alg{A}$ to compact elements of $\Fi^{\theta}_{\logic{L}} \alg{A}$.

  Given any $a_{1}, \dots, a_{n}, b \in \alg{A}$, we can now consider the family of congruences
\begin{align*}
  \set{\theta \in \Con_{\class{K}} \alg{A}}{b \in \Fg^{\alg{A}, \theta}_{\logic{L}}(a_{1}, \dots, a_{n})}.
\end{align*}
  This family of congruences is an upset of $\Con_{\class{K}} \alg{A}$. It is non-empty for each $n \geq 1$, and if $\logic{L}$ has a theorem, then it is non-empty also for $n = 0$. Such families of congruences on the formula algebra $\Fm_{n+1}$ over $n+1$ variables will provide us with the required sequence~$\boldPsi_{n}$.

\begin{theorem} \label{thm: ledcf}
  Let $\class{K} \supseteq \Alg \logic{L}$ be a prevariety. Then $\logic{L}$ has a local EDCF on $\class{K}$ if and only if it has the absolute FEP on $\class{K}$.
\end{theorem}

\begin{proof}
  Suppose that $\logic{L}$ has the local EDCF($\boldPsi$) on $\class{K}$. Consider $\alg{A} \leq \alg{B}$ with $\alg{B} \in \class{K}$ and $a_{1}, \dots, a_{n}, b \in \alg{A}$. Since $\class{K}$ is a prevariety, $\alg{A} \in \class{K}$. Then
\begin{align*}
  b \in \Fg^{\alg{B}}_{\logic{L}}(a_{1}, \dots, a_{n}) & \iff \alg{B} \vDash \boldTheta_{n}(a_{1}, \dots, a_{n}, b) \text{ for some } \boldTheta_{n} \in \boldPsi_{n} \\
  & \iff \alg{A} \vDash \boldTheta_{n}(a_{1}, \dots, a_{n}, b) \text{ for some } \boldTheta_{n} \in \boldPsi_{n} \\
  & \iff b \in \Fg^{\alg{A}}_{\logic{L}}(a_{1}, \dots, a_{n}).  
\end{align*}

  Conversely, let $\Fm_{n+1}$ be the algebra of formulas of $\logic{L}$ in the ${n+1}$ variables $\{ x_{1}, \dots, x_{n}, y \}$, i.e.\ the absolutely free algebra generated by this set in the signature of $\logic{L}$, and let $\boldPsi_{n}(x_{1}, \dots, x_{n}, y)$ be the upset of all $\class{K}$-congruences $\theta$ on $\Fm_{n+1}$ such that $y \in \Fg^{\Fm_{n+1}, \theta}_{\logic{L}}(x_{1}, \dots, x_{n})$. Observe that $\boldPsi_{n}$ is non-empty unless $n = 0$ and $\logic{L}$ has no theorem. We claim that $\boldPsi_{n}$ is a local equational definition of $n$-generated $\logic{L}$-filters.

  Consider an algebra $\alg{B} \in \class{K}$ and $a_{1}, \dots, a_{n}, b \in \alg{B}$. Let $\alg{A}$ be the subalgebra of $\alg{B}$ generated by these elements. $\class{K}$ is closed under subalgebras, so $\alg{A} \in \class{K}$. Let $h\colon \Fm_{n+1} \to \alg{A}$ be the unique homomorphism with $h(x_{1}) = a_{1}$, \dots, $h(x_{n}) = a_{n}$, and $h(y) = b$. This homomorphism is surjective, so by the First Isomorphism Theorem~\cite[Theorem~6.12]{burris+sankappanavar81} the algebras $\alg{A}$ and $\Fm_{n+1} / \Ker h$ are isomorphic via $h(x) \mapsto \pi(x)$, where $\pi\colon \Fm_{n+1} \to \Fm_{n+1} / \Ker h$ is the quotient map. We have the following chain of equivalences for each $\boldTheta_{n} \in \boldPsi_{n}$:
\begin{align*}
  \alg{B} \vDash \boldTheta_{n}(a_{1}, \dots, a_{n}, b) & \iff \alg{A} \vDash \boldTheta_{n}(a_{1}, \dots, a_{n}, b) \\
  & \iff \alg{A} \vDash \boldTheta_{n}(h(x_{1}), \dots, h(x_{n}), h(y)) \\
  & \iff \Fm_{n+1} / \Ker h \vDash \boldTheta_{n}(\pi(x_{1}), \dots, \pi(x_{n}), \pi(y)) \\
  & \iff \boldTheta_{n}(x_{1}, \dots, x_{n}, y) \leq \Ker h.
\end{align*}
  Since $\boldPsi_{n}$ is an upset, it follows that $\alg{B} \vDash \boldTheta_{n}(a_{1}, \dots, a_{n}, b)$ for some $\boldTheta_{n} \in \boldPsi_{n}$ if and only if $\Ker h \in \boldPsi_{n}$. But
\begin{align*}
  \Ker h \in \boldPsi_{n} & \iff y \in \Fg^{\Fm_{n}, \Ker h}_{\logic{L}}(x_{1}, \dots, x_{n}) \\
  & \iff \pi(y) \in \Fg^{\Fm_{n} / \Ker h}_{\logic{L}}(\pi(x_{1}), \dots, \pi(x_{n})) \\
  & \iff h(y) \in \Fg^{\alg{A}}_{\logic{L}}(h(x_{1}), \dots, h(x_{n})) \\
  & \iff b \in \Fg^{\alg{A}}_{\logic{L}}(a_{1}, \dots, a_{n}) \\
  & \iff b \in \Fg^{\alg{B}}_{\logic{L}}(a_{1}, \dots, a_{n}),
\end{align*}
  where the last equivalence uses the absolute FEP on $\class{K}$.
\end{proof}

  The following is a less refined but more concise form of the above theorem.

\begin{theorem} \label{thm: baby ledcf}
  Suppose that $\Alg \logic{L}$ is closed under subalgebras. Then $\logic{L}$ has a local EDCF on $\Alg \logic{L}$ if and only if it has the absolute FEP on $\Alg \logic{L}$.
\end{theorem}

  As examples of logics with a local EDCF we have already seen the global modal logic $\GlobK$ and {\L}ukasiewicz logic $\Luk$, as well as their fragments in any signature containing $\Box, \wedge, 1$ in the case of $\GlobK$ and in any signature containing $\odot, 1$ in the case of $\Luk$. Many other global modal logics or substructural logics have a local EDCF of the same form. A trivial example is the identity logic $\IdL$ where $\Gamma \vdash_{\IdL} \varphi$ if and only if $\varphi \in \Gamma$, and accordingly $b \in \Fg^{\alg{A}}_{\IdL}(a_{1}, \dots, a_{n})$ if and only if $b = a_{1}$ or \dots\ or $b = a_{n}$.

  A more exotic example is provided by the paraconsistent weak Kleene logic $\PWK$ introduced already in the previous section (recall Example~\ref{example: pwk does not have pedcf}). To establish the local EDCF for $\PWK$, we shall first need the following description of $\PWK$-filters on generalized involutive bisemilattices, i.e.\ algebras in the variety $\GIB$ generated by $\WKthree$.

\begin{lemma}[{\cite[Proposition~7.1.8 and Remark 7.1.9]{bonzio+paoli+prabaldi22}}]\label{lemma: pwk filters}
  Let $\alg{A}\in \GIB$. A set $F\subseteq\alg{A}$ is a $\PWK$-filter on $\alg{A}$ if and only if the following hold for every $a,b\in \alg{A}$:
\begin{enumerate}[(i)]
 \item $a\vee\neg a\in F$,
 \item if $a\in F$ and $a\leq_{\vee} b$, then $b\in F$,
 \item if $a,b\in F$, then $a\wedge b\in F$,
 \end{enumerate}
  where $\leq_{\vee}$ is the partial order determined by the join operation $\vee$:
\begin{align*}
  a \leq_{\vee} b \iff a \vee b = b.
\end{align*}
  The smallest $\PWK$-filter on $\alg{A}$ is the set $P(\alg{A}) \assign \set{a \in \alg{A}}{a = a \vee \neg a}$.
 \end{lemma}

  Recall that $\WKthree$ ordered by $\leq_{\vee}$ is the three-element chain $0 < 1 < \onehalf$.

\begin{example}\label{example: pwk local edcf}
   The logic $\PWK$ has a local EDCF on $\GIB \supseteq \Alg \PWK$, namely
\begin{align*}
  \boldPsi_{n} \assign \{ \{ y \equals y \vee \neg y \} \} \cup \bigcup \set{\{ \{ \bigwedge X \leq_{\vee} y \} \}}{X \subseteq \{ x_{1}, \dots, x_{n} \}}.
\end{align*}
  More explicitly, $b \in \Fg^{\alg{A}}_{\PWK}(a_{1}, \dots, a_{n})$ for $a_{1}, \dots, a_{n}, b \in \alg{A} \in \GIB$ if and only if
\begin{align*}
  \alg{A} \vDash b \equals b \vee \neg b \text{ or } c_{1} \wedge \dots \wedge c_{k} \leq_{\vee} b \text{ for some } \{ c_{1}, \dots, c_{k} \} \subseteq \{ a_{1}, \dots, a_{n} \}.
\end{align*}
\end{example}

\begin{proof} 
  Consider $\alg{A} \in \GIB$ and $a_{1},\dots,a_{n}\in\alg{A}$. Take
\begin{align*}
  F \assign \set{b \in \alg{A}}{\alg{A} \vDash \boldTheta_{n}(a_{1}, \dots, a_{n}, b) \text{ for some } \boldTheta_{n}\in\boldPsi_{n}}.
\end{align*}
  We show that $F=\Fg^{\alg{A}}_{\PWK}(a_{1},\dots,a_{n})$. Clearly $a_{i}\in F$ for each $i \in \{ 1, \dots, n \}$, since $\leq_{\vee}$ is reflexive. Next, we prove that $F$ is a $\PWK$-filter. To this end it suffices to verify that $F$ satisfies conditions (i)--(iii) of Lemma~\ref{lemma: pwk filters}.

  Condition (i) holds, i.e.\ $P(\alg{A}) \subseteq F$, because $\{ x \equals x \vee \neg x \} \in \boldPsi_{n}$ and $\GIB \vDash x \vee \neg x \equals (x \vee \neg x) \vee \neg (x \vee \neg x)$.

  In order to show that condition (ii) holds, i.e.\ that $F$ is an upset with respect to $\leq_{\vee}$, suppose $a\in F$ and $a \leq_{\vee} b \in \alg{A}$. Then either $a \in P(\alg{A})$ (i.e.\ $a = a \vee \neg a$) or $c_{i} \wedge \dots \wedge c_{k} \leq_{\vee} a$ for some $\{ c_{1}, \dots, c_{k} \} \subseteq \{a_{1},\dots,a_{n} \}$.
   In the former case, if $a \in P(\alg{A})$ and $a \leq_{\vee} b$, then $b = b \vee \neg b \in P(\alg{A})$ by the last claim in Lemma~\ref{lemma: pwk filters}, therefore $b \in F$, since $\{ y \equals y \vee \neg y \} \in \boldPsi_{n}$.
   In the latter case, $a_{i}\wedge\dots\wedge a_{k}\leq_{\vee} a \leq_{\vee} b$, so $b\in F$, since $\{x_{i}\wedge\dots\wedge x_{k}\leq_{\vee} y\}\in\boldPsi_{n}$.

  Finally, to prove condition that condition (iii) holds, i.e.\ that $F$ is closed under $\wedge$, consider $a,b\in F$. 
   If $a, b \in P(\alg{A})$, then $a\wedge b\in P(\alg{A})$ by the last claim in Lemma~\ref{lemma: pwk filters}, therefore $a \wedge b \in F$. 
   If only one among the elements $a$ and $b$, say $a$, lies in $P(\alg{A})$, then $c_{1} \wedge \dots \wedge c_{k} \leq_{\vee} b$ for some non-empty set $\{ c_{1},\dots, c_{k}\} \subseteq \{a_{1},\dots,a_{n}\}$.
  Since $\GIB \vDash x \leq_{\vee} x \wedge (x \vee y)$ and $\GIB \vDash (x \vee \neg x) \wedge y \equals y \wedge (y\vee (x \vee \neg x))$, we obtain that $c_{1} \wedge \dots \wedge c_{k} \leq_{\vee} b \leq_{\vee} b\wedge (b \vee a) = a\wedge b \in F$.
  If neither $a$ nor $b$ lie in $P(\alg{A})$, then $c \assign c_{1}\wedge\dots\wedge c_{k} \leq_{\vee} a$ and $d \assign d_{1}\wedge\dots\wedge d_{l}\leq_{\vee} b$ for some non-empty sets $\{c_{1},\dots,c_{k}\},\{d_{1},\dots,d_{l}\}\subseteq\{a_{1},\dots,a_{n}\}$. Observe that $\GIB \vDash x \wedge y \leq_{\vee} (x \vee u) \wedge (y \vee v)$, so the inequalities $c \leq_{\vee} a$ and $d \leq_{\vee} b$ imply that $c \wedge d \leq_{\vee} a \wedge b$. But $\{ c_{1}, \dots, c_{k}, d_{1}, \dots, d_{l} \}$ is a non-empty subset of $\{ a_{1}, \dots, a_{n} \}$, so $c \wedge d \in F$. Because we have already proved that $F$ is an upset with respect to $\leq_{\vee}$, it follows that $a \wedge b \in F$. This proves that $F$ is indeed a $\PWK$-filter.

  It remains to show that $F$ is the smallest $\PWK$-filter with $a_{1}, \dots, a_{n} \in F$. Let $G$ therefore be a $\PWK$-filter containing $a_{1},\dots,a_{n}$ and consider $a \in F$. If $a$ is a solution of $x \equals x \vee \neg x$ on $\alg{A}$, then $a\in G$ by item (i) of Lemma~\ref{lemma: pwk filters}. Otherwise, $c_{1} \wedge \dots \wedge c_{k}\leq a$ for some non-empty finite set $\{ c_{1}, \dots, c_{k} \} \subseteq \{ a_{1}, \dots, a_{n} \}$. Since $G$ contains $c_{1},\dots, c_{k}$, item (iii) of Lemma~\ref{lemma: pwk filters} entails that $c_{1} \wedge \dots \wedge c_{k}\in G$. Moreover, by condition (ii) the inequality $c_{1} \wedge \dots \wedge c_{k} \leq_{\vee} a$ entails that $a \in G$. Therefore $F\subseteq G$, which proves that $F$ is indeed the smallest $\PWK$-filter on $\alg{A}$ containing $a_{1}, \dots, a_{n}$.
\end{proof}

  We now show under what conditions a local EDCF can be pared down to a global EDCF.

\begin{lemma} \label{lemma: directed unions}
  Let $\theta \in \Con_{\class{K}} \alg{A}$ be the directed union of a family of $\class{K}$-congruences $\theta_{i} \in \Con_{\class{K}} \alg{A}$ for $i \in I$. Then the $\logic{L}$-filter $\Fg^{\alg{A}, \theta}_{\logic{L}}(a_{1}, \dots, a_{n})$ is the directed union of the family of $\logic{L}$-filters $\Fg^{\alg{A}, \theta_{i}}_{\logic{L}}(a_{1}, \dots, a_{n})$ for $i \in I$.
\end{lemma}

\begin{proof}
  Let $F$ be the directed union of this family of $\logic{L}$-filters. Then $F$ is compatible with $\theta_{i}$ for each $i \in I$ by virtue of being a directed union of $\logic{L}$-filters compatible with $\theta_{i}$. It follows that $F$ is compatible with $\theta$. Conversely, each $\logic{L}$-filter $G$ compatible with $\theta$ is compatible with each $\theta_{i}$ for $i \in I$, so $\Fg^{\alg{A}, \theta_{i}}_{\logic{L}}(a_{1}, \dots, a_{n}) \subseteq G$ for each $i \in I$, and thus $F \subseteq G$.
\end{proof}

\begin{theorem} \label{thm: edcf} \label{thm: EDCF=FEP+cong}
  Let $\logic{L}$ be a logic with the local EDCF($\boldPsi$) on a prevariety~$\class{K} \supseteq \Alg \logic{L}$. Then the following are equivalent:
\begin{enumerate}[(i)]
\item $\logic{L}$ has an EDCF on $\class{K}$.
\item $\logic{L}$ has the EDCF($\boldTheta$) on $\class{K}$ for some sequence $\boldTheta$ with $\boldTheta_{n} \in \boldPsi_{n}$ for all~$n$.
\item For each $\alg{A}$ (or equivalently, for each $\alg{A} \in \class{K}$) and compact $F, G \in \Fi_{\logic{L}} \alg{A}$ there is a smallest $\class{K}$-congruence $\theta$ on $\alg{A}$ such that $G \subseteq \Fg^{\alg{A}, \theta}_{\logic{L}} F$.
\item For each $\alg{A}$ (or equivalently, for each $\alg{A} \in \class{K}$) and $a_{1}, \dots, a_{n}, b \in \alg{A}$ there is a smallest $\class{K}$-congruence $\theta$ on $\alg{A}$ such that $b \in \Fg^{\alg{A}, \theta}_{\logic{L}}(a_{1}, \dots, a_{n})$.
\item $\logic{L}$ has factor determined compact filters on $\class{K}$.
\end{enumerate}
  If $\class{K}$ is a quasivariety, then this $\class{K}$-congruence $\theta$ is a compact $\class{K}$-congruence.
\end{theorem}

\begin{proof}
  (i) $\Leftrightarrow$ (v): having an EDCF on $\class{K}$ is equivalent to having both a local EDCF and a parametrized EDCF on $\class{K}$ by Corollary~\ref{cor: global edcf}, and the latter is equivalent to having factor determined compact filters on $\class{K}$ by Theorem~\ref{thm: pedcf}.

  (i) $\Leftrightarrow$ (ii): this was already proved in Theorem~\ref{thm: removing locality}.

  (iii)$\Rightarrow$(iv): it suffices to take $F \assign \Fg^{\alg{A}}_{\logic{L}}(a_{1}, \dots, a_{n})$ and $G \assign \Fg^{\alg{A}}_{\logic{L}}(b)$ and to observe that for each $\logic{L}$-filter $H$ on $\alg{A}$ we have $b \in H$ if and only if $G \subseteq H$.

  (iv)$\Rightarrow$(iii): Consider $F \assign \Fg^{\alg{A}}_{\logic{L}}(a_{1}, \dots, a_{m})$ and $G \assign \Fg^{\alg{A}}_{\logic{L}} (b_{1}, \dots, b_{n})$. Let $\theta_{i}$ for $1 \leq i \leq n$ be the smallest congruence such that $b_{i} \in \Fg^{\alg{A}, \theta_{i}}(a_{1}, \dots, a_{n})$. Then for each $\class{K}$-congruence $\theta$
\begin{align*}
  G \subseteq \Fg^{\alg{A}, \theta} F & \iff b_{i} \in \Fg^{\alg{A}, \theta}(a_{1}, \dots, a_{n}) \text{ for each } 1 \leq i \leq n \\
  & \iff \theta_{i} \leq \theta \text{ for each } 1 \leq i \leq n \\
 & \iff \theta_{1} \vee \dots \vee \theta_{n} \leq \theta,
\end{align*}
  so $\theta_{1} \vee \dots \vee \theta_{n}$ is the smallest congruence $\theta$ such that $G \subseteq \Fg^{\alg{A}, \theta} F$.

  (i)$\Rightarrow$(iv): Suppose that $\logic{L}$ has the EDCF($\boldTheta$) with respect to $\class{K}$. Consider an algebra $\alg{A} \in \class{K}$ with $a_{1}, \dots, a_{n}, b \in \alg{A}$ and a $\class{K}$-congruence $\theta \in \Con_{\class{K}} \alg{A}$. Then
\begin{align*}
  b \in \Fg^{\alg{A}, \theta}_{\logic{L}}(a_{1}, \dots, a_{n}) & \iff b / \theta \in \Fg^{\alg{A} / \theta}_{\logic{L}}(a_{1} / \theta, \dots, a_{n} / \theta) \\
  & \iff \alg{A} / \theta \vDash \boldTheta_{n}(a_{1} / \theta, \dots, a_{n} / \theta, b / \theta) \\
  & \iff \boldTheta_{n}(a_{1}, \dots, a_{n}, b) \subseteq \theta \\
  & \iff \Cg^{\alg{A}}_{\class{K}} \boldTheta_{n}(a_{1}, \dots, a_{n}, b) \leq \theta \text{ in } \Con_{\class{K}} \alg{A},
\end{align*}
  where the third equivalence holds because $\alg{A} / \theta \in \class{K}$. Thus $\Cg^{\alg{A}}_{\class{K}} \boldTheta_{n}(a_{1}, \dots, a_{n}, b)$ is the smallest $\class{K}$-congruence on $\alg{A}$ with the required property.

  (iv)$\Rightarrow$(i): Let $\Fm_{n+1}$ be the algebra of formulas of $\logic{L}$ over the set of variables $\{ x_{1}, \dots, x_{n}, y \}$. By assumption there is a smallest $\theta \in \Con_{\class{K}} \Fm_{n+1}$ such that $y \in \Fg^{\Fm_{n+1}, \theta}_{\logic{L}}(x_{1}, \dots, x_{n})$, i.e.\ $y / \theta \in \Fg^{\Fm_{n+1} / \theta}(x_{1} / \theta, \dots, x_{n} / \theta)$. We take $\boldTheta_{n}(x_{1}, \dots, x_{n}) \assign \theta$. (Note that $\theta$ is indeed a set of pairs of elements of $\Fm_{n+1}$.)

 We claim that $\logic{L}$ has the EDCF($\boldTheta$). Consider $a_{1}, \dots, a_{n}, b \in \alg{A} \in \class{K}$. Let $\alg{B}$ be the subalgebra of $\alg{A}$ generated by these elements and let ${h\colon \Fm_{n+1} \to \alg{B}}$ be the unique (surjective) homomorphism such that $h(x_{1}) = a_{1}$, \dots, $h(x_{n}) = a_{n}$, and $h(y) = b$. Then
\begin{align*}
  b \in \Fg^{\alg{A}}_{\logic{L}} (a_{1}, \dots, a_{n}) & \iff b \in \Fg^{\alg{B}}_{\logic{L}} (a_{1}, \dots, a_{n}) \\
  & \iff y \in \Fg^{\Fm_{n+1}, \Ker h} (x_{1}, \dots, x_{n}) \\
  & \iff \theta \leq \Ker h \\
  & \iff \boldTheta_{n}(x_{1}, \dots, x_{n}) \subseteq \Ker h \\ 
  & \iff \Fm_{n+1} / \Ker h \vDash \boldTheta_{n}(x_{1} / \Ker h, \dots, x_{n} / \Ker h) \\
  & \iff \alg{B} \vDash \boldTheta_{n}(a_{1}, \dots, a_{n}, b), \\
  & \iff \alg{A} \vDash \boldTheta_{n}(a_{1}, \dots, a_{n}, b),
\end{align*}
  where the first equivalence uses the absolute FEP, the second relies on the surjectivity of $h$, and the third and fourth use the definitions of $\theta$ and $\boldTheta_{n}$.

  (ii)$\Rightarrow$(i): Trivial. 

  Finally, suppose that $\class{K}$ is a quasivariety and $G \subseteq \Fg^{\alg{A}, \theta}_{\logic{L}} F$. Then $\theta$ is a directed union of a family of compact $\class{K}$-congruences $\theta_{i} \leq \theta$ with $i \in I$, so $\Fg^{\alg{A}, \theta}_{\logic{L}} F$ is a directed union of $\Fg^{\alg{A}, \theta_{i}}_{\logic{L}} F$ with $i \in I$ by Lemma~\ref{lemma: directed unions}. Because $G$ is finitely generated, it follows that $G \subseteq \Fg^{\alg{A}, \theta_{i}}_{\logic{L}} F$ for some $i \in I$. But $\theta$ is the smallest congruence such that $G \subseteq \Fg^{\alg{A}, \theta}_{\logic{L}} F$, so in fact $\theta = \theta_{i}$.
\end{proof}

  The following is a less refined but more concise form of the above theorem.

\begin{theorem} \label{thm: baby edcf}
  Let $\logic{L}$ be a logic such that $\Alg \logic{L}$ is closed under subalgebras. Then the following are equivalent:
\begin{enumerate}[(i)]
\item $\logic{L}$ has an EDCF on $\Alg \logic{L}$.
\item $\logic{L}$ has the absolute FEP and for each $\alg{A} \in \Alg \logic{L}$ and compact $F, G \in \Fi_{\logic{L}} \alg{A}$ there is a smallest $(\Alg \logic{L})$-congruence $\theta$ on $\alg{A}$ such that $G \subseteq \Fg^{\alg{A}, \theta}_{\logic{L}} F$.
\item $\logic{L}$ has the absolute FEP and for each $\alg{A} \in \Alg \logic{L}$ and $a_{1}, \dots, a_{n}, b \in \alg{A}$ there is a smallest $(\Alg \logic{L})$-congruence $\theta$ on $\alg{A}$ such that $b \in \Fg^{\alg{A}, \theta}_{\logic{L}}(a_{1}, \dots, a_{n})$.
\end{enumerate}
  If $\Alg \logic{L}$ is closed under ultraproducts, then this $(\Alg \logic{L})$-congruence $\theta$ is compact.
\end{theorem}

  Echoing the remarks made at the end of the previous section, the absolute FEP is similar to but distinct from the FEP (filter extension property), which is the semantic correlate of the local DDT~\cite[p.~138]{czelakowski01}. We again postpone the discussing the relationship between these two properties to Section~\ref{sec: edcf and ddt}.

\begin{fact} \label{fact: extending edcf}
  Let $\boldTheta$ be a global EDCF for a logic $\logic{L}$ on a class of algebras~$\class{K}$. If there are formulas $\varphi_{1}, \dots, \varphi_{k+1}$ and $\psi_{1}, \dots, \psi_{k}, \psi_{k+1} \assign y$ with $k \in \omega$ such that
\begin{align*}
  & x_{1}, \dots, x_{n}, \psi_{1}, \dots, \psi_{i} \vdash_{\logic{L}} \varphi_{i+1} & & \text{for all } i \in \{ 0, \dots, k \}
\end{align*}
  and
\begin{align*}
  & \class{K} \vDash \bigwedge \boldTheta_{n}(x_{1}, \dots, x_{n}) \implies \varphi_{i} \equals \psi_{i} & & \text{for all } i \in \{ 1, \dots, k+1 \},
\end{align*}
  then $\boldTheta$ is a global equational definition of compact $\logic{L}$-filters on $\AlgI \AlgS \AlgP(\class{K})$.
\end{fact}

\begin{proof}
  Given $a_{1}, \dots, a_{n}, b \in \alg{A} \in \AlgI \AlgS \AlgP (\class{K})$, let $F \assign \Fg^{\alg{A}}_{\logic{L}} (a_{1}, \dots, a_{n})$. Because the validity of the above implication is preserved under isomorphic images, subalgebras, and products,
\begin{align*}
  \AlgI \AlgS \AlgP (\class{K}) \vDash \bigwedge \boldTheta_{n}(x_{1}, \dots, x_{n}) \implies \varphi_{i} \equals \psi_{i}.
\end{align*}
  If $\alg{A} \vDash \boldTheta_{n}(a_{1}, \dots, a_{n}, b)$, then applying any homomorphism $h\colon \Fm \to \alg{A}$ such that $h(x_{1}) = a_{1}, \dots, h(x_{n}) = a_{n}, h(y) = b$ to the rules assumed to be valid yields that
\begin{align*}
  h(\varphi_{1}) \in F \implies h(\psi_{1}) \in F \implies h(\varphi_{2}) \in F \implies \dots \implies h(\psi_{k+1}) \in F,
\end{align*}
  so $b = h(y) = h(\psi_{k+1}) \in F$.

  Conversely, $\logic{L}$ has some parametrized local EDCF on the class of all algebras (in the given signature) by Theorem~\ref{thm: pledcf}, say $\boldPsi'$. Because $\boldTheta$ is a global EDCF for $\logic{L}$ on $\class{K}$, for each $\boldTheta'_{n} \in \boldPsi'_{n}$
\begin{align*}
  \class{K} \vDash \bigwedge \boldTheta'_{n}(x_{1}, \dots, x_{n}, y, \tuple{z}) \implies \bigwedge \boldTheta_{n}(x_{1}, \dots, x_{n}, y).
\end{align*}
  This implication extends to $\AlgI \AlgS \AlgP(\class{K})$. But then the parametrized local EDCF($\boldPsi'$) ensures that if $b \in F$, then $\alg{A} \vDash \boldTheta_{n}(a_{1}, \dots, a_{n}, b)$.
\end{proof}

  As an example, consider the strong three-valued Kleene logic $\KL$, or more precisely its variant with the top and bottom constants. This is a logic in the signature which consists of the binary lattice connectives $\wedge$ and $\vee$, the top and bottom constants $1$ and $0$, and a unary negation connective $\neg$. It is determined by the matrix $\langle \Kthree, \{ 1 \} \rangle$, where $\Kthree$ is the three-element bounded chain $0 < \onehalf < 1$ with the operation $\neg\colon x \mapsto 1 - x$.

\begin{example} \label{example: kleene has edcf}
  $\KL$ has a global EDCF on $\Alg \KL = \AlgH \AlgS \AlgP(\Kthree)$, namely
\begin{align*}
  \boldTheta_{n}(x_{1}, \dots, x_{n}, y) \assign \{ x_{1} \wedge \dots \wedge x_{n} \leq \neg x_{1} \vee \dots \vee \neg x_{n} \vee y \}.
\end{align*}
\end{example}

\begin{proof}
  The only non-trivial $\KL$-filter on $\Kthree$ is $\{ 1 \}$, so $\boldTheta$ is easily seen to be a global EDCF for $\KL$ on $\Kthree$. Moreover, taking $\varphi \assign x_{1} \wedge \dots \wedge x_{n}$,
\begin{align*}
  \Kthree \vDash \bigwedge \boldTheta_{n}(x_{1}, \dots, x_{n}, y) \implies \varphi \equals \varphi \wedge (\neg \varphi \vee y).
\end{align*}
  But then the rules
\begin{align*}
  & x_{1}, \dots, x_{n} \vdash_{\KL} \varphi, & & \varphi \wedge (\neg \varphi \vee y) \vdash_{\KL} y,
\end{align*}
  ensure that $\boldTheta$ is an EDCF for $\KL$ on $\AlgI \AlgS \AlgP(\Kthree)$ by Fact~\ref{fact: extending edcf}, taking $\class{K} \assign \{ \Kthree \}$.

  The class $\Alg \KL$ is known to be the variety generated by $\Kthree$. Let us give a brief self-contained proof of this fact. Because $\Kthree$ is a finite algebra with a lattice reduct, by J\'{o}nsson's lemma each subdirectly irreducible algebra in $\AlgH \AlgS \AlgP(\Kthree)$ lies in $\AlgH \AlgS \AlgPU (\Kthree) = \AlgI \AlgS (\Kthree)$. But each algebra in $\AlgH \AlgS \AlgP(\Kthree)$ is a subdirect product of subdirectly irreducible algebras in this variety, so $\AlgH \AlgS \AlgP(\Kthree) = \AlgI \AlgS \AlgP(\Kthree)$. In addition, $\Alg \KL \subseteq \AlgH \AlgS \AlgP(\Kthree)$ by Fact~\ref{fact: alg l variety}. Conversely, $\Alg \KL$ is closed under subdirect products and each algebra in $\AlgI \AlgS \AlgP(\Kthree)$ is a subdirect product of algebras in $\AlgS(\Kthree)$. Thus $\AlgI \AlgS \AlgP(\Kthree) \subseteq \Alg \KL$, since $\AlgS(\Kthree) \subseteq \Alg \KL$.
\end{proof}

  The same argument with the four-element subdirectly irreducible De~Morgan algebra $\DMfour$ instead of $\Kthree$ shows that the four-valued Exactly True Logic $\ETL$ (see~\cite{albuquerque+prenosil+rivieccio17,prenosil23}), or more precisely its variant with the top and bottom constants, has EDCF($\boldTheta$) with respect to $\Alg \ETL = \AlgI \AlgS \AlgP(\DMfour)$. This class is known as the variety of De~Morgan algebras, while its subvariety $\AlgI \AlgS \AlgP(\Kthree)$ is known as the variety of Kleene algebras.

  A similar example but with a different EDCF is provided by the Logic of Paradox $\LP$, or more precisely its variant with the top and bottom constants. This is the logic determined by the matrix $\langle \Kthree, \{ \onehalf, 1 \} \rangle$.

\begin{example}
  $\LP$ has a global EDCF on $\Alg \LP = \AlgH \AlgS \AlgP(\Kthree)$, namely
\begin{align*}
  \boldTheta_{n}(x_{1}, \dots, x_{n}, y) \assign \{ x_{1} \wedge \dots \wedge x_{n} \wedge \neg y \leq y \},
\end{align*}
  where the case of $n \assign 0$ is interpreted as $\boldTheta_{0}(y) \assign \{ \neg y \leq y \}$.
\end{example}

\begin{proof}
  The only non-trivial $\LP$-filter on $\Kthree$ is $\{ \onehalf, 1 \}$, so $\boldTheta$ is easily seen to be a global EDCF for $\LP$ on $\Kthree$. Moreover, taking again $\varphi \assign x_{1} \wedge \dots \wedge x_{n}$,
\begin{align*}
  \Kthree \vDash \bigwedge \boldTheta_{n}(x_{1}, \dots, x_{n}, y) \implies y \equals (\varphi \wedge \neg y) \vee y.
\end{align*}
  But then the rule
\begin{align*}
  & x_{1}, \dots, x_{n} \vdash_{\LP} (\varphi \wedge \neg y) \vee y
\end{align*}
  ensures that $\boldTheta$ is an EDCF for $\LP$ on $\AlgI \AlgS \AlgP(\Kthree)$ by Fact~\ref{fact: extending edcf}, taking $\class{K} \assign \{ \Kthree \}$. In more detail, this rule is derivable from the rules
\begin{align*}
  & \emptyset \vdash_{\LP} y \vee \neg y, & & x, y \vdash_{\LP} x \wedge y, & & x \wedge (y \vee z) \vdash_{\LP} (x \wedge y) \vee z.
\end{align*}
  The claim that $\Alg \LP = \AlgH \AlgS \AlgP(\Kthree)$ is proved as in the previous example.
\end{proof}

  The top and bottom constants are in fact typically not taken to be part of the signature of $\KL$, $\ETL$, and $\LP$, perhaps for reasons of tradition. Removing them from $\KL$ and $\ETL$ results in logics which almost have a global EDCF. In contrast, the above EDCF for $\LP$ does not rely on the constants.

\section{Finitary EDCF}
\label{sec: finitary edcf}

  The various forms of the EDCF considered so far define the $(n+1)$-ary relation $b \in \Fg^{\alg{A}}_{\logic{L}}(a_{1}, \dots, a_{n})$ in terms of conditions beyond the expressive power of first-order logic: the parametrized EDCF allows for existential quantification over infinitely many variables, the local EDCF allows for infinite disjunctions, and all forms of the EDCF allow for infinite conjunctions of equations. In this section, we consider the case where the relation is in fact definable by a first-order formula. This will have the effect of finitizing the above disjunctions and conjunctions (and therefore also the existential quantification).

\begin{definition}
  A logic $\logic{L}$ has \emph{(first-order) definable $n$-generated filters} on a class of algebras $\class{K}$ (for a given $n \in \omega$) if there is a first-order formula in $n+1$ free variables $\varphi_{n}(x_{1}, \dots, x_{n}, y)$ such that
\begin{align*}
  b \in \Fg^{\alg{A}}_{\logic{L}}(a_{1}, \dots, a_{n}) \iff \alg{A} \vDash \varphi_{n}(a_{1}, \dots, a_{n}, b).
\end{align*}
  The logic $\logic{L}$ has \emph{definable compact filters} on $\class{K}$ if it has definable $n$-generated filters on $\class{K}$ for each $n \in \omega$. It has \emph{definable principal filters}, or \emph{DPF} for short, on $\class{K}$ if it has definable $1$-generated filters on $\class{K}$.
\end{definition}

  The first-order definability of principal filters was studied by Czelakowski~\cite[p.~132]{czelakowski01}, who observed that the definability of $1$-generated filters implies the definability of $n$-generated filters for all $n \geq 1$~\cite[Theorem~2.2.1]{czelakowski01}. We can further observe that it implies the definability of $0$-generated filters too: in case $\logic{L}$ has no theorems take $\varphi_{0}(x) \assign \bot$, and in case $\logic{L}$ has a theorem take $\varphi_{0}(x) \assign \forall y \; \varphi_{1}(y, x)$. Czelakowski's DPF is therefore precisely the first-order counterpart of our EDCF.

\begin{definition}
  A parametrized local EDCF $\boldPsi$ is \emph{finitary} if each $\boldPsi_{n}$ for $n \in \omega$ is a finite family of finite sets of formulas.
\end{definition}

  Because the local EDCF, the parametrized EDCF, and the global EDCF are defined as special cases of the the parametrized local EDCF, the above definition also covers the finitary local, parametrized, and global EDCF.

  While the global EDCF is frequently finitary, typically the local EDCF is not (much like the global DDT is frequently finitary, but typically the local DDT is not). However, examples of logics with a finitary local EDCF but not a global EDCF do exist: the paraconsistent Weak Kleene logic is one (Example~\ref{example: pwk local edcf}).

  The restriction to finitary forms of the EDCF allows us to relate our present results about the EDCF to the work of Campercholi and Vaggione~\cite{campercholi+vaggione16}, which in effect answers the following question: when does a first-order definable relation admit a definition of a particularly simple syntactic form? The syntactic forms considered by their paper in particular include forms which correspond directly to the finitary EDCF and its local, parametrized, and parametrized local versions. The following theorem is an amalgam of some of the main results of~\cite{campercholi+vaggione16} restricted to classes closed under ultraproducts.

  (In the theorem, we emphasize that homomorphisms of structures are understood in the standard model theoretic sense of maps which preserve functions and relations. In particular, they need not be strict homomorphisms.)

\begin{theorem} \label{thm: campercholi vaggione}
   Let $\lang{L}_{R}$ be the expansion of a first-order language $\lang{L}$ by a relation symbol $R$, and let $\class{K}_{R}$ be a class of $\lang{L}_{R}$-structures closed under ultraproducts such that each homomorphism $h\colon \mathbb{A} \to \mathbb{B}$ between the $R$-free reducts of $\lang{L}_{R}$-structures $\mathbb{A}_{R}, \mathbb{B}_{R} \in \class{K}_{R}$ is a homomorphism of $\lang{L}_{R}$-structures $h\colon \mathbb{A}_{R} \to \mathbb{B}_{R}$. Then:
\begin{enumerate}[(i)]
\item $R$ is definable in $\class{K}_{R}$ by a formula $\varphi$ in $\lang{L}$ of the form $\exists \bigvee \bigwedge \At$.
\item If $\class{K}_{R}$ is closed under substructures, we can take $\varphi$ of the form $\bigvee \bigwedge \At$.
\item If $\class{K}_{R}$ is closed under products, we can take $\varphi$ of the form $\exists \bigwedge \At$.
\item If $\class{K}_{R}$ is closed under both constructions, we can take $\varphi$ of the form $\bigwedge \At$.
\end{enumerate}
\end{theorem}

\begin{theorem} \label{thm: finitary edcf}
  A logic has the finitary EDCF on a quasivariety $\class{K}$ if and only if has both the EDCF and the DPF on $\class{K}$. The same equivalence holds for the parametrized, local, and parametrized local forms of the EDCF.
\end{theorem}

\begin{proof}
 The left-to-right implications are clear, in particular observe that if $\logic{L}$ has the finitary parametrized local EDCF on $\class{K}$ with respect to~$\boldPsi$, then it has the DPF on $\class{K}$ with respect to
\begin{align*}
  \varphi_{n}(x_{1}, \dots, x_{n}, y) \assign \exists \tuple{z} \bigvee \set{\bigwedge \boldTheta_{n}(x_{1}, \dots, x_{n}, y, \tuple{z})}{\boldTheta_{n} \in \boldPsi_{n}}.
\end{align*}
  Conversely, let $\lang{L}$ be the algebraic signature of the given logic $\logic{L}$, and let $\lang{L}_{R}$ be the expansion of $\lang{L}$ by a relation symbol $R$ of arity $n+1$ for $n \in \omega$. Let $\class{K}_{R}$ be the class of expansions of algebras in $\class{K}$ by the relation $R$ such that
\begin{align*}
  \langle a_{1}, \dots, a_{n}, b \rangle \in R \iff b \in \Fg^{\alg{A}}_{\logic{L}} (a_{1}, \dots, a_{n}).
\end{align*}
  We now verify that the assumptions of Theorem~\ref{thm: campercholi vaggione} are satisfied. If the logic $\logic{L}$ has the DPF on $\class{K}$, then the class $\class{K}_{R}$ is closed under ultraproducts by {\L}o\'{s}'s theorem on the satisfaction of first-order formulas in ultraproducts. The condition that each homomorphism $h\colon \alg{A} \to \alg{B}$ of algebras $\alg{A}, \alg{B} \in \class{K}$ is in fact a homomorphism of structures $h\colon \pair{\alg{A}}{R} \to \pair{\alg{B}}{R}$ is equivalent to the condition that $b \in \Fg^{\alg{A}}_{\logic{L}} (a_{1}, \dots, a_{n})$ implies $h(b) \in \Fg^{\alg{A}}_{\logic{L}} (h(a_{1}), \dots, h(a_{n}))$ for all ${a_{1}, \dots, a_{n}, b \in \alg{A}}$. This holds because if $h(b) \notin \Fg^{\alg{A}}_{\logic{L}} (h(a_{1}), \dots, h(a_{n}))$, then there is some $\logic{L}$-filter $G$ on $\alg{B}$ with $h(a_{1}), \dots, h(a_{n}) \in G$ but $h(b) \notin G$, which yields an $\logic{L}$-filter $F \assign h^{-1}[G]$ on $\alg{A} \in \Fi_{\logic{L}}$ with $a_{1}, \dots, a_{n} \in F$ but $b \notin F$, thus $b \notin \Fg^{\alg{A}}_{\logic{L}} (a_{1}, \dots, a_{n})$.

  Theorem~\ref{thm: campercholi vaggione} now yields a formula $\varphi$ of the form $\exists \bigvee \bigwedge \At$ in the signature $\lang{L}$ which gives us the finitary EDCF. If $\logic{L}$ has the local EDCF, then by (the easy direction of) Theorem~\ref{thm: ledcf} the class $\class{K}_{R}$ is closed under substructures, so by Theorem~\ref{thm: campercholi vaggione} we can take $\varphi$ of the form $\bigvee \bigwedge \At$. Similarly, if $\logic{L}$ has the parametrized EDCF, then by (the easy direction of) Theorem~\ref{thm: pedcf} the class $\class{K}_{R}$ is closed under products, so by Theorem~\ref{thm: campercholi vaggione} we can take $\varphi$ of the form $\exists \bigwedge \At$. Finally, if $\logic{L}$ has the EDCF, both closure conditions apply, so we can take $\varphi$ of the form $\bigwedge \At$.
\end{proof}

  Notice that for the finitary case we could also derive our characterizations of the different forms of the EDCF from Theorem~\ref{thm: campercholi vaggione}. For example, item (ii) of this theorem corresponds directly to our equivalence between the local EDCF and the absolute FEP (restricted to logics with DPF).

\begin{corollary}
  If $\logic{L}$ has the DPF and the parametrized local EDCF($\boldPsi$) on a quasivariety $\class{K}$, then it has the parametrized local EDCF($\boldPsi'$) for some sequence $\boldPsi'$ such that each $\boldPsi'_{n}$ is finite and each $\boldTheta'_{n} \in \boldPsi'_{n}$ is a finite subset of some $\boldTheta_{n} \in \boldPsi_{n}$.
\end{corollary}

\begin{proof}
  By the previous theorem, $\logic{L}$ has a finitary parametrized local EDCF for some $\boldPsi^{\mathrm{fin}}$. Fact~\ref{fact: edcf unique} tells us that for each of the finitely many finite sets $\boldTheta^{\mathrm{fin}}_{n} \in \boldPsi^{\mathrm{fin}}_{n}$ there is some $\boldTheta''_{n} \in \boldPsi_{n}$ such that
\begin{align*}
  \class{K} \vDash \exists \tuple{z} \bigwedge \boldTheta^{\mathrm{fin}}_{n}(x_{1}, \dots, x_{n}, y, \tuple{z}) \implies \exists \tuple{z} \bigwedge \boldTheta'_{n}(x_{1}, \dots, x_{n}, y, \tuple{z}).
\end{align*}
  These sets yield a finite subfamily $\boldPsi''_{n}$ of $\boldPsi_{n}$ such that $\logic{L}$ has the parametrized local EDCF($\boldPsi''$). Again by Fact~\ref{fact: edcf unique}, for each $\boldTheta''_{n} \in \boldPsi''_{n}$ there is some (finite) $\boldTheta^{\mathrm{fin}}_{n} \in \boldPsi^{\mathrm{fin}}_{n}$ such that
\begin{align*}
  \class{K} \vDash \exists \tuple{z} \bigwedge \boldTheta''_{n}(x_{1}, \dots, x_{n}, y, \tuple{z}) \implies \exists \tuple{z} \bigwedge \boldTheta^{\mathrm{fin}}_{n}(x_{1}, \dots, x_{n}, y, \tuple{z}),
\end{align*}
  or equivalently
\begin{align*}
  \class{K} \vDash \bigwedge \boldTheta''_{n}(x_{1}, \dots, x_{n}, y, \tuple{z}) \implies \exists \tuple{z} \bigwedge \boldTheta^{\mathrm{fin}}_{n}(x_{1}, \dots, x_{n}, y, \tuple{z}).
\end{align*}
  Since each $\boldTheta^{\mathrm{fin}}_{n} \in \boldPsi^{\mathrm{fin}}_{n}$ is finite, by the compactness of first-order logic there is some finite subset $\boldTheta'_{n} \subseteq \boldTheta''_{n}$ such that
\begin{align*}
  \class{K} \vDash \bigwedge \boldTheta'_{n}(x_{1}, \dots, x_{n}, y, \tuple{z}) \implies \exists \tuple{z} \bigwedge \boldTheta^{\mathrm{fin}}_{n}(x_{1}, \dots, x_{n}, y, \tuple{z}).
\end{align*}
  Consequently, $\alg{A} \vDash \boldTheta_{n}(a_{1}, \dots, a_{n}, b, \tuple{c})$ implies that $b \in \Fg^{\alg{A}}_{\logic{L}} (a_{1}, \dots, a_{n})$. Replacing each set $\boldTheta''_{n}$ in the finite family $\boldPsi''_{n}$ by the finite set $\boldTheta'_{n}$ therefore yields a family $\boldPsi'_{n}$ with the required properties.
\end{proof}

\section{The Leibniz hierarchy}
\label{sec: leibniz}

  Before discussing the relationship between the EDCF and the DDT, we shall first need to review certain important families of logics, which form part of the so-called Leibniz hierarchy of abstract algebraic logic. The reader familiar with the topic can skip this section without loss of continuity. The reader looking for more details can consult the textbook~\cite{font16} or the monograph~\cite{czelakowski01}.

  The Leibniz hierarchy is so called because its levels can be defined in terms of the behavior of the \emph{Leibniz operator} on $\logic{L}$-filters. This is the map
\begin{align*}
  \Leibniz^{\alg{A}}\colon \Fi_{\logic{L}} \alg{A} \to \Con_{\Alg \logic{L}} \alg{A}
\end{align*}
  which assigns to each $\logic{L}$-filter $F$ on an algebra $\alg{A}$ the largest congruence on $\alg{A}$ compatible with $F$. (Recall that compatibility means that $a \in F$ and $\pair{a}{b} \in \theta$ implies $b \in F$.) Such a congruence exists for each $F$ and the quotient algebra always lies in $\Alg \logic{L}$. This congruence is called the \emph{Leibniz congruence} of $F$. We shall use $\alg{A} / F$ and $a / F$ as abbreviations for $\alg{A} / \Leibniz^{\alg{A}} F$ and $a / \Leibniz^{\alg{A}} F$, respectively.

  A logic $\logic{L}$ is \emph{protoalgebraic} if the Leibniz operator is monotone on the $\logic{L}$-filters of each algebra $\alg{A}$ (without loss of generality in $\Alg \logic{L}$):
\begin{align*}
  F \subseteq G \text{ for } F, G \in \Fi_{\logic{L}} \alg{A} \implies \Leibniz^{\alg{A}} F \leq \Leibniz^{\alg{A}} G.
\end{align*}
  Equivalently, a logic is protoalgebraic if there is a set of formulas in two variables $\Delta(x, y)$, called a \emph{proto\-implication set}, such that
\begin{align*}
  & \emptyset \vdash_{\logic{L}} \Delta(x, x), & & x, \Delta(x, y) \vdash_{\logic{L}} y.
\end{align*}
  In particular, any logic with a binary implication connective $\rightarrow$ such that
\begin{align*}
  & \emptyset \vdash_{\logic{L}} x \rightarrow x, & & x, x \rightarrow y \vdash_{\logic{L}} y.
\end{align*}
  is protoalgebraic. The most common logics which have an implication connective generally validate this two rules. A third equivalent characterization is that $\logic{L}$ is proto\-algebraic if it has a \emph{set of congruence formulas with parameters}: a set of formulas $\Delta(x, y, \tuple{z})$ such that for each algebra $\alg{A}$ (without loss of generality in $\Alg \logic{L}$) and each $\logic{L}$-filter $F$ on $\alg{A}$
\begin{align*}
  \pair{a}{b} \in \Leibniz^{\alg{A}} F \iff \Delta(a, b, \tuple{c}) \subseteq F \text{ for each tuple $\tuple{c}$ of elements of $\alg{A}$}.
\end{align*}

   A logic $\logic{L}$ is \emph{(finitely) equivalential} if it has a (finite) set of congruence formulas $\Delta(x, y)$ which do not involve parameters, or equivalently a (finite) set of formulas $\Delta(x, y)$ which is a protoimplication set and moreover for each $n$-ary connective $f$ in the signature of $\logic{L}$
\begin{align*}
  \Delta(x_{1}, y_{1}), \dots, \Delta(x_{n}, y_{n}) \vdash_{\logic{L}} \Delta(f(x_{1}, \dots, x_{n}), f(y_{1}, \dots, y_{n})).
\end{align*}
  The local consequence relation of basic modal logic, denoted here by $\LocK$, is an example of (non-algebraizable) equivalential logic.

  A logic $\logic{L}$ is \emph{weakly algebraizable} if it is protoalgebraic and moreover the Leibniz operator is injective on the $\logic{L}$-filters of each algebra $\alg{A}$ (without loss of generality in $\Alg \logic{L}$). In that case, the Leibniz operator on $\logic{L}$-filters is in fact an isomorphism between $\Fi_{\logic{L}}$ and $\Con_{\Alg \logic{L}} \alg{A}$ for each $\alg{A}$.

  Finally, a logic $\logic{L}$ is \emph{algebraizable} if it is weakly algebraizable and equivalential. A non-trivially equivalent definition is that $\logic{L}$ is algebraizable if it is proto\-algebraic, the Leibniz operator on $\logic{L}$-filters is injective, and moreover it commutes with homomorphisms: for each homomorphism $h\colon \alg{A} \to \alg{B}$ (without loss of generality between algebras in $\Alg \logic{L}$) and each $\logic{L}$-filter $F$ on $\alg{B}$
\begin{align*}
  \Leibniz^{\alg{A}} h^{-1}[F] = h^{-1}[\Leibniz^{\alg{B}} F].
\end{align*}
  A third equivalent characterization of algebraizable logics is that they are weakly algebraizable logics $\logic{L}$ such that $\Alg \logic{L}$ is closed under sub\-algebras (this follows from~\cite[Theorem~6.73]{font16}).

  The following simple observations will be useful in the next section. Note that the product of a family of congruences $\theta_{i}$ of algebras $\alg{A}_{i}$ with $i \in I$ is defined as the congruence $\prod_{i \in I} \theta_{i}$ on $\alg{A} \assign \prod_{i \in I} \alg{A}_{i}$ such that $\pair{a}{b} \in \prod_{i \in I} \theta_{i}$ for $a, b \in \alg{A}$ if and only if $\pair{\pi_{i}(a)}{\pi_{i}(b)} \in \theta_{i}$ for all $i \in I$.

\begin{lemma} \label{lemma: leibniz of product}
  Let $\logic{L}$ be a protoalgebraic logic. Consider $\alg{A}_{i} \in \Alg \logic{L}$ and $F_{i} \in \Fi_{\logic{L}} \alg{A}_{i}$, with $\alg{A} \assign \prod_{i \in I} \alg{A}_{i}$ and $F \assign \prod_{i \in I} F_{i}$. Then $\Leibniz^{\alg{A}} F = \prod_{i \in I} \Leibniz^{\alg{A}_{i}} F_{i}$. That is, $a / F = b / F$ for $a, b \in \alg{A}$ if and only if $\pi_{i}(a) / F_{i} = \pi_{i}(b) / F_{i}$ for each $i \in I$.
\end{lemma}

\begin{proof}
  We use the fact that $\logic{L}$ has a set of congruence formulas with parameters $\Delta(x, y, \tuple{z})$:
\begin{align*}
  \pair{a}{b} \in \Leibniz^{\alg{A}} F & \iff \Delta(a, b, \tuple{c}) \subseteq F \text{ for each tuple $\tuple{c}$ in $\alg{A}$} \\
  & \iff \Delta(\pi_{i}(a), \pi_{i}(b), \pi_{i}(\tuple{c})) \subseteq F_{i} \text { for each $i \in I$ and each $\tuple{c} \in \alg{A}$} \\
  & \iff \Delta(\pi_{i}(a), \pi_{i}(b), \tuple{d}) \subseteq F_{i} \text { for each $i \in I$ and each $\tuple{d} \in \alg{A}_{i}$} \\
  & \iff \pair{\pi_{i}(a)}{\pi_{i}(b)} \in \Leibniz^{\alg{A}_{i}} F_{i} \text{ for each $i \in I$} \\
  & \iff \pair{a}{b} \in \prod_{i \in I} \Leibniz^{\alg{A}_{i}} F_{i}. \qedhere
\end{align*}  
\end{proof}

\begin{lemma} \label{lemma: weakly algebraizable quotient}
  Let $\logic{L}$ be a weakly algebraizable logic. Consider an algebra $\alg{A} \in \Alg \logic{L}$, an $\logic{L}$-filter $F$ on $\alg{A}$, and $a_{1}, \dots, a_{n}, b \in \alg{A}$. Then $b \in \Fg^{\alg{A}}_{\logic{L}}(F, a_{1}, \dots, a_{n})$ if and only if $b / F \in \Fg^{\alg{A} / F}_{\logic{L}}(a_{1} / F, \dots, a_{n} / F)$.
\end{lemma}

\begin{proof}
  If $G$ is an $\logic{L}$-filter on $\alg{A}$ which extends $F$, then by protoalgebraicity $\Leibniz^{\alg{A}} F \leq \Leibniz^{\alg{A}} G$, so $H \assign \set{a / F}{a \in G}$ is an $\logic{L}$-filter on $\alg{A} / F$ and $G = \pi^{-1}[H]$, where ${\pi\colon \alg{A} \to \alg{A} / F}$ is the quotient map. Consequently, $b \in \Fg^{\alg{A}} (F, a_{1}, \dots, a_{n})$ if and only if $b / F \in \Fg^{\alg{A} / F} (\pi[F], a_{1} / F, \dots, a_{n} / F)$. Because $\logic{L}$ is weakly algebraizable, $\Leibniz^{\alg{A}} F \leq \Leibniz^{\alg{A}} F'$ for $F' \in \Fi_{\logic{L}} \alg{A}$ if and only if $F \subseteq F'$. It follows that $\pi[F]$ is the smallest $\logic{L}$-filter on $\alg{A} / F$, and thus $\Fg^{\alg{A} / F} (\pi[F], a_{1} / F, \dots, a_{n} / F) = \Fg^{\alg{A} / F} (a_{1} / F, \dots, a_{n} / F)$.
\end{proof}

\section{The EDCF and the DDT}
\label{sec: edcf and ddt}

  As foreshadowed in the preceding sections, there is a close relationship between the different forms of the EDCF and the corresponding forms of the DDT. The goal of this section is to clarify this relationship.

  Firstly, we show that for algebraizable logics the two hierarchies coincide. Secondly, we show that in practice each form of the DDT often implies the corresponding form of the EDCF, in particular this happens under the modest assumption that $0$-generated filters are equationally definable. Finally, we show that in full generality the DDT does not imply the EDCF, and conversely we construct a protoalgebraic logic with the EDCF but without the DDT. This proves that although the semantic correlates of the corresponding forms of the DDT and the EDCF look very similar, none of them imply the other.

\begin{definition}
  A logic $\logic{L}$ has the \emph{parametrized local DDT} if there is a family $\ddtfamily$ of sets of formulas $\ddtset(x, y, \tuple{z})$, where $\tuple{z}$ is a possibly infinite tuple of variables and the variables $x, y, \tuple{z}$ are distinct, such that
\begin{align*}
  \Gamma, \varphi \vdash_{\logic{L}} \psi \iff \Gamma \vdash_{\logic{L}} \ddtset(\varphi, \psi, \tuple{\chi}) \text{ for some } \ddtset \in \ddtfamily \text{ and some tuple of formulas } \tuple{\chi}. 
\end{align*}
  If the variables $\tuple{z}$ do not occur in any of the sets $\ddtset$, we say that $\logic{L}$ has the \emph{local DDT}. If $\ddtfamily$ contains exactly one set of formuas $\ddtset$, we say that $\logic{L}$ has the \emph{parametrized DDT}. If both of these conditions hold, it has the \emph{(global) DDT}.
\end{definition}

  In the parametrized local case, there is nothing to discuss: we have seen that every logic has the parametrized local EDCF (Theorem~\ref{thm: pledcf}), while the parametrized local DDT is equivalent to protoalgebraicity.

  The semantic correlates of the parametrized DDT are the following.

\begin{definition}
  A logic $\logic{L}$ has \emph{factor determined compact filters} if for each family of algebras $\alg{A}_{i}$ with $i \in I$, each family of filters $F_{i} \in \Fi_{\logic{L}} \alg{A}_{i}$, and $\alg{A} \assign \prod_{i \in I} \alg{A}_{i}$ and $F \assign \prod_{i \in I} F_{i}$ we have
\begin{align*}
  \Fg^{\alg{A}}_{\logic{L}}(F, a_{1}, \dots, a_{n}) = \prod_{i \in I} \Fg^{\alg{A}_{i}}_{\logic{L}}(F_{i}, \pi_{i}(a_{1}), \dots, \pi_{i}(a_{n})).
\end{align*}
\end{definition}

  Observe that if a logic $\logic{L}$ has factor determined compact filters and moreover for each family of algebras $\alg{A}_{i}$ with $i \in I$
\begin{align*}
  \Fg^{\alg{A}}_{\logic{L}} \emptyset = \prod_{i \in I} \Fg^{\alg{A}_{i}}_{\logic{L}} \emptyset,
\end{align*}
  then (taking $F_{i} \assign \Fg^{\alg{A}_{i}}_{\logic{L}}$) it has absolute factor determined compact filters.

\begin{definition}
  An \emph{$n$-test matrix for compact filters} of $\logic{L}$ on a sub\-direct class $\class{K}$ is a matrix $\pair{\alg{A}_{n}}{F_{n}} \in \class{K}$ with elements $p_{1}, \dots, p_{n}, q \in \alg{A}_{n}$ (called \emph{test elements}) such that for each $\alg{A} \in \class{K}$
\begin{align*}
  b \in \Fg^{\alg{A}}_{\logic{L}}(a_{1}, \dots, a_{n}) & \iff \text{there is a homomorphism } h\colon \alg{A}_{n} \to \alg{A} \text{ with } \\
  & \phantom{\iff} ~ h(p_{1}) = a_{1}, \dots, h(p_{n}) = a_{n}, \text{ and } h(q) = b.
\end{align*}
  Equivalently, $q \in \Fg^{\alg{A}_{n}}_{\logic{L}}(p_{1}, \dots, p_{n})$ and moreover $b \in \Fg^{\alg{A}}_{\logic{L}}(a_{1}, \dots, a_{n})$ implies the existence of the above homomorphism $h$. A logic $\logic{L}$ \emph{has test matrices for compact filters} on $\class{K}$ if it has an $n$-test matrix on $\class{K}$ for all~$n \in \omega$.
\end{definition}

\begin{theorem}[{\cite[Thm.~2.4.1]{czelakowski01}}]
  The following are equivalent for each logic~$\logic{L}$:
\begin{enumerate}[(i)]
\item $\logic{L}$ has the parametrized DDT.
\item $\logic{L}$ is protoalgebraic and has factor determined compact filters.
\item $\logic{L}$ is protoalgebraic and has test matrices.
\end{enumerate}
\end{theorem}

\begin{theorem} \label{thm: pddt = pedcf}
  A weakly algebraizable logic has factor determined compact filters if and only if it has absolute factor determined compact filters. It thus has the parametrized DDT if and only if it has the parametrized EDCF.
\end{theorem}

\begin{proof}
  Let $\logic{L}$ be a weakly algebraizable logic. Suppose first that $\logic{L}$ has factor determined filters. To show that $\logic{L}$ has absolute factor determined filters, it suffices to show that
\begin{align*}
  \Fg^{\alg{A}}_{\logic{L}} \emptyset = \prod_{i \in I} \Fg^{\alg{A}_{i}}_{\logic{L}} \emptyset.
\end{align*}
  for each family of algebras $\alg{A}_{i} \in \Alg \logic{L}$ with $i \in I$. By weak algebraizability it suffices to show that the Leibniz congruences of the two filters coincide. But
\begin{align*}
  \Leibniz^{\alg{A}} (\Fg^{\alg{A}}_{\logic{L}} \emptyset) = \Delta_{\alg{A}} = \prod_{i \in I} \Delta_{\alg{A}_{i}} = \prod_{i \in I} \Leibniz^{\alg{A}_{i}} (\Fg^{\alg{A}_{i}}_{\logic{L}} \emptyset) = \Leibniz^{\alg{A}} \left( \prod_{i \in I} \Fg^{\alg{A}_{i}}_{\logic{L}} \emptyset \right),
\end{align*}
  where we use Lemma~\ref{lemma: leibniz of product} and the fact that $\Leibniz^{\alg{B}} (\Fg^{\alg{B}}_{\logic{L}} \emptyset) = \Delta_{\alg{B}}$ for each $\alg{B} \in \Alg \logic{L}$ by protoalgebraicity.

  Conversely, suppose that $\logic{L}$ has absolute factor determined filters. Consider a family of algebras $\alg{A}_{i} \in \Alg \logic{L}$ with $i \in I$ and a family of $\logic{L}$-filters $F_{i}$ on~$\alg{A}_{i}$, and take $\alg{A} \assign \prod_{i \in I} \alg{A}_{i}$ and $F \assign \prod_{i \in I} F_{i}$. For $a_{1},\dots,a_{n},b\in \alg{A}$ we need to show that
\begin{align*}
  \pi_{i}(b)\in \Fg^{\alg{A}_{i}}_{\logic{L}} (F_{i}, \pi_{i}(a_{1}), \dots, \pi_{i}(a_{n})) \text{ for each }i\in I\end{align*}
 if and only if 
 \begin{align*}
  b\in\Fg^{\alg{A}}_{\logic{L}} (F, a_{1}, \dots, a_{n}).
\end{align*}
   Let $\alg{B} \assign \prod_{i \in I} \alg{A}_{i} / F_{i}$. By Lemma~\ref{lemma: leibniz of product} the map $a / F \mapsto (\pi_{i}(a) / F_{i})_{i \in I}$ is a well-defined embedding $\alg{A} / F \to \alg{B}$. Since each $(\pi_{i}(a)/F_{i})_{i\in I} \in \alg{B}$ is the image of $(\pi_{i}(a))_{i \in I} / F$, this map is also surjective, thus it is an isomorphism.
   
   The above isomorphism, together with Lemma~\ref{lemma: weakly algebraizable quotient}, reduces our task to showing that
\begin{align*}
  \pi_{i}(b) / F_{i} \in \Fg^{\alg{A}_{i} / F_{i}} (\pi_{i}(a_{1}) / F_{i}, \dots, \pi_{i}(a_{n}) / F_{i}) \text{ for each }i \in I
\end{align*}
  if and only if
\begin{align*}
  (\pi_{i}(b)/F_{i})_{i\in I} \in \Fg^{\alg{B}}_{\logic{L}} ((\pi_{i}(a_{1}) / F_{i})_{i \in I}, \dots, (\pi_{i}(a_{n}) / F_{i})_{i \in I}).
\end{align*}
 But this equivalence holds because $\logic{L}$ has absolute factor determined filters.
\end{proof}

  The semantic correlate of the local DDT is the following.

\begin{definition}
  A logic $\logic{L}$ has the \emph{filter extension property (FEP)} if for each submatrix $\pair{\alg{A}}{F}$ of a model $\pair{\alg{B}}{G}$ of $\logic{L}$ each $\logic{L}$-filter $F' \supseteq F$ on $\alg{A}$ is the restriction to $\alg{A}$ of some $\logic{L}$-filter $G \subseteq G$ on $\alg{B}$.
\end{definition}

  In case $\Alg \logic{L}$ is a prevariety, observe that if $\logic{L}$ has the FEP and moreover
\begin{align*}
  \Fg^{\alg{A}}_{\logic{L}} \emptyset = \alg{A} \cap \Fg^{\alg{B}}_{\logic{L}} \emptyset \text{ for each } \alg{A} \leq \alg{B} \in \Alg \logic{L},
\end{align*}
  then $\logic{L}$ has the absolute FEP (take $F = \Fg^{\alg{A}}_{\logic{L}} \emptyset$ and $G \assign \Fg^{\alg{B}}_{\logic{L}} \emptyset$ in the FEP).

\begin{theorem}[{\cite[Thm.~2.3.5]{czelakowski01}}] \label{thm: local ddt}
  A logic has the local DDT if and only if it is proto\-algebraic and has the FEP.
\end{theorem}

\begin{theorem} \label{thm: lddt = ledcf}
  An algebraizable logic has the FEP if and only if it has the absolute FEP. It thus has the local DDT if and only if it has the local EDCF.
\end{theorem}

\begin{proof}
  Let $\logic{L}$ be an algebraizable logic. Suppose first that $\logic{L}$ has the FEP. To prove that it has the absolute FEP, it suffices to show that given algebras $\alg{A} \leq \alg{B} \in \Alg \logic{L}$ we have $\Fg^{\alg{A}}_{\logic{L}} \emptyset = \alg{A} \cap \Fg^{\alg{B}}_{\logic{L}} \emptyset$. To this end, it suffices to show that $\Leibniz^{\alg{A}} \Fg^{\alg{A}}_{\logic{L}} \emptyset = \Leibniz^{\alg{A}} (\alg{A} \cap \Fg^{\alg{B}}_{\logic{L}} \emptyset)$. Observe that $\Leibniz^{\alg{A}} \Fg^{\alg{A}}_{\logic{L}} \emptyset = \Delta_{\alg{A}}$ and $\Leibniz^{\alg{B}} \Fg^{\alg{B}}_{\logic{L}} \emptyset = \Delta_{\alg{B}}$, since the Leibniz operator is an isomorphism between $\logic{L}$-filters and $(\Alg \logic{L})$-congruences. Since $\logic{L}$ has a set of congruence formulas without parameters $\Delta(x, y)$, for each $a, b \in \alg{A}$ we have
\begin{align*}
  \pair{a}{b} \in \Leibniz^{\alg{A}} (\alg{A} \cap \Fg^{\alg{B}}_{\logic{L}} \emptyset) & \iff \Delta(a, b) \subseteq \alg{A} \cap \Fg^{\alg{B}}_{\logic{L}} \emptyset \\
  & \iff \Delta(a, b) \subseteq \Fg^{\alg{B}} \emptyset \\
  & \iff \pair{a}{b} \in \Leibniz^{\alg{B}}_{\logic{L}} \Fg^{\alg{B}} \emptyset \\
  & \iff \pair{a}{b} \in \Delta_{\alg{B}} \\
  & \iff \pair{a}{b} \in \Delta_{\alg{A}} = \Leibniz^{\alg{A}} \Fg^{\alg{A}}_{\logic{L}} \emptyset,
\end{align*}
  where the second equivalence holds because $\Delta(a, b) \subseteq \alg{A}$.

  Conversely, suppose that $\logic{L}$ has the absolute FEP and consider a submatrix $\pair{\alg{A}}{F}$ of a model $\pair{\alg{B}}{G}$ of $\logic{L}$ and an $\logic{L}$-filter $F' \supseteq F$ on $\alg{A}$. Because $\logic{L}$ has a set of congruence formulas without parameters, $\Leibniz^{\alg{A}} F = \alg{A}^{2} \cap \Leibniz^{\alg{B}} G$, and thus $\alg{A} / F \leq \alg{B} / G$. Let $\pi\colon \alg{B} \to \alg{B} / G$ be the quotient map. Because $\logic{L}$ is protoalgebraic and $F' \supseteq F$, the congruence $\Leibniz^{\alg{A}} F$ is compatible with $F'$ and thus $\pi[F']$ is an $\logic{L}$-filter on $\alg{A} / F$. By the absolute FEP, $F'$ extends to an $\logic{L}$-filter $H'$ on $\alg{B} / G$. Now consider the $\logic{L}$-filter $G' \assign \pi^{-1}[H']$ on $\alg{B}$. Because $\pi[G]$ is the smallest filter on $\alg{B}  / G$, it follows that $G' \supseteq G$. Because $\pi[F']$ is the restriction of $H'$ to $\alg{A} / F$, it follows that $F' = \pi^{-1}[\pi[F']]$ is the restriction of $G' = \pi^{-1}[H']$ to $\alg{A}$.
\end{proof}

  The following observation is well known.

\begin{fact} \label{fact: global ddt}
  A logic has the global DDT if and only if it has both the parametrized and the local DDT.
\end{fact}

\begin{proof}
  Let $\ddtfamily$ and $\ddtset'$ be a family of sets of formulas and a set of formulas witnessing the local DDT and the parametrized DDT for a logic $\logic{L}$. Then $x, \ddtset'(x, y, \tuple{z}) \vdash_{\logic{L}} y$, so $\ddtset'(x, y, \tuple{z}) \vdash \ddtset(x, y)$ for some $\ddtset \in \ddtfamily$, and by structurality $\ddtset'(\varphi, \psi, \tuple{\alpha}) \vdash_{\logic{L}} \ddtset(\varphi, \psi)$ for each tuple of formulas $\tuple{\alpha}$. But then $\Gamma, \varphi \vdash_{\logic{L}} \psi$ implies that $\Gamma \vdash_{\logic{L}} \ddtset'(\varphi, \psi, \tuple{\alpha})$, so $\Gamma \vdash_{\logic{L}} \ddtset(\varphi, \psi)$. The set $\ddtset(x, y)$ therefore shows that $\logic{L}$ has the global DDT.
\end{proof}

\begin{corollary} \label{cor: ddt = edcf}
  An algebraizable logic has the global DDT if and only if it has the global EDCF.
\end{corollary}

\begin{proof}
  This follows from the equivalence for the local and the parametrized DDT and EDCF, the fact that the global DDT is the conjunction of the local and the parametrized DDT (Fact~\ref{fact: global ddt}), and the fact that the same holds for the EDCF (Corollary~\ref{cor: global edcf}).
\end{proof}

  A more interesting semantic correlate of the global DDT is stated in terms of compact filters. Recall that an $\logic{L}$-filter $F$ on an algebra $\alg{A}$ is \emph{compact} if has the form $F = \Fg^{\alg{A}}_{\logic{L}} X$ for some finite $X \subseteq \alg{A}$, or equivalently if it is a compact element of the lattice $\Fi_{\logic{L}} \alg{A}$ of all $\logic{L}$-filters on $\alg{A}$.

\begin{theorem}[{\cite[Theorem~6.28]{font16}}] \label{thm: global ddt}
  A protoalgebraic logic has the global DDT if and only if the join semilattice of compact $\logic{L}$-filters on each algebra $\alg{A}$ is dually Brouwerian: for all compact $F, G \in \Fi_{\logic{L}} \alg{A}$ there is a smallest compact $H \in \Fi_{\logic{L}} \alg{A}$ such that $G \leq F \vee H$, where the join takes place in the lattice $\Fi_{\logic{L}} \alg{A}$, or equivalently in the join semilattice of compact $\logic{L}$-filters on $\alg{A}$.
\end{theorem}

  It is instructive to compare the condition in the above theorem to condition (iii) in Theorem~\ref{thm: edcf} characterizing the EDCF: informally speaking, we have replaced the filter $G$ (but not the filter $F$) by a congruence.

  To be able to deduce the EDCF from the DDT, it suffices to assume that the logic in question satisfies the EDCF$_{0}$, the equational definability of $0$-generated filters. Recall that $\logic{L}$ enjoys EDCF$_{0}$ on a class $\class{K}$ if the smallest $\logic{L}$-filter on each $\alg{A} \in \class{K}$, i.e.\ the filter $\Fg^{\alg{A}}_{\logic{L}} \emptyset$, is definable by some set of equations in one variable $\boldTheta_{0}(x)$. The parametrized and local forms are defined as expected.

\begin{fact} \label{fact: ddt implies edcf}
  If $\logic{L}$ enjoys the (finitary) DDT and the (finitary) EDCF$_{0}$ on a class $\class{K}$, then it also enjoys the (finitary) EDCF on $\class{K}$. The same implication holds if we replace all of these properties by their parametrized or local forms.
\end{fact}

\begin{proof}
  We only consider the global case. Suppose that DDT is witnessed by the set $\ddtset(x, y)$ and the EDCF$_{0}$ by the set $\boldTheta_{0}(x)$. Then
\begin{align*}
  b \in \Fg^{\alg{A}}_{\logic{L}} (a_{1}, \dots, a_{n}) & \iff \ddtset_{n}(a_{1}, \dots, a_{n}, b) \subseteq \Fg^{\alg{A}}_{\logic{L}} \emptyset \\
  & \iff \alg{A} \vDash \boldTheta_{0}(\ddtformula(a_{1}, \dots, a_{n}, b)) \text{ for each } \ddtformula \in \ddtset_{n}.
\end{align*}
  If $\ddtset$ and $\boldTheta_{0}$ are finite sets, then so is $\bigcup \set{\boldTheta_{0}(\alpha(x_{1}, \dots, x_{n}, y)}{\alpha \in \ddtset_{n}}$.
\end{proof}

  The EDCF$_{0}$ is a very common property in practice, even among logics which lack the EDCF: often the smallest $\logic{L}$-filter on algebras in $\Alg \logic{L}$ is the principal upset generated by some constant $1$ and is therefore defined by one of the equations $1 \wedge x \equals x$ or $1 \vee x \equals x$ or $1 \equals x$.

  Nonetheless, we now construct an example of a logic which enjoys the DDT but not the EDCF. The example relies on the observation that the DDT is preserved under adding axioms to a logic and adding extra operations to its signature, while the EDCF is not.

  The signature of our example expands the signature of intuitionistic logic (containing in particular the top constant $1$) by two unary operators, $\circ$ and~$\bullet$. The axiomatization of this logic $\ILbc$ (over some infinite set of variables $\Var$) consists of an axiomatization of intuitionistic logic plus the axiom $\emptyset \vdash \circ x$. This logic is non-trivial, i.e.\ $\emptyset \nvdash_{\ILbc} x$ for $x \in \Var$. (This is witnessed e.g.\ by the two-element Heyting algebra with $0 \leq 1$ expanded by $\circ\colon a \mapsto 1$ and $\bullet\colon a \mapsto 1$ for $a \in \{ 0, 1 \}$.)

\begin{lemma}
  $\bullet \varphi \dashv \vdash_{\ILbc} \bullet \psi$ if and only if $\varphi = \psi$.
\end{lemma}

\begin{proof}
    The right-to-left implication is trivial. Conversely, consider distinct formulas $\varphi$ and $\psi$. We shall assume without loss of generality that $\varphi$ is not a subformula of $\psi$, and prove under this assumption that $\bullet \varphi \nvdash_{\ILbc} \bullet \psi$.

  Let $\Fm_{\varphi}$ be the algebra obtained from $\Fm$ by redefining the operation $\bullet$ at exactly one point, namely $\bullet\colon \varphi \mapsto 1$, and keeping all other operations the same. In particular, for each homomorphism $h\colon \Fm \to \Fm_{\varphi}$ there is a homomorphism $h'\colon \Fm \to \Fm$ such that $h'(\alpha) = h(\alpha)$ whenever $\alpha$ does not contain $\bullet$, namely the homomorphism extending of the map $x \mapsto h(x)$ for $x \in \Var$. Since $\ILbc$ is axiomatized by rules which do not contain $\bullet$, this implies that each theory of $\ILbc$ is a filter on $\Fm_{\varphi}$. In more detail, consider a rule $\Gamma \vdash \varphi$ valid in $\ILbc$ which does not contain $\bullet$. If $T$ is a theory of $\ILbc$ and $h[\Gamma] \subseteq T$ for some homomorphism ${h\colon \Fm \to \Fm_{\varphi}}$, then $h'[\Gamma] \subseteq T$ for this homomorphism $h'\colon \Fm \to \Fm$, therefore $h'(\varphi) \in T$ and $h(\varphi) \in T$.

  The identity map on $\Var$ extends to a homomorphism ${h\colon \Fm \to \Fm_{\varphi}}$. Observe that $h(\alpha) = \alpha$ for every formula $\alpha$ which does not contain $\bullet \varphi$ as a subformula. In particular, $h(\bullet \psi) = \bullet \psi$. On the other hand, $h(\bullet \varphi) = \bullet^{\Fm_{\varphi}} h(\varphi) = \bullet^{\Fm_{\varphi}} \varphi = 1$. Take $T$ to be the set of theorems of $\ILbc$. If $\emptyset \nvdash_{\ILbc} \bullet \psi$, then $T$ is a filter of $\ILbc$ on $\Fm_{\varphi}$ such that $h(\bullet \varphi) \in T$ and $h(\bullet \psi) \notin T$, so $\bullet \varphi \nvdash_{\ILbc} \bullet \psi$ as desired.

  It remains to show that $\emptyset \nvdash_{\ILbc} \bullet \psi$. Take $\Fm'_{\psi}$ to be the algebra obtained from $\Fm$ by redefining the operation $\bullet$ at exactly one point, namely $\bullet\colon \psi \mapsto x$ for some $x \in \Var$. Again, each theory of $\ILbc$ and in particular the set of theorems~$T$, is a filter of $\ILbc$ on $\Fm'_{\psi}$. The map $x \mapsto x$ for $x \in \Var$ again extends to a homomorphism $g\colon \Fm \to \Fm'_{\psi}$. This time $g(\bullet \psi) = x$. Because $\ILbc$ is not a trivial logic, $x \notin T$, so $g$ witnesses that $\emptyset \nvdash_{\ILbc} \bullet \psi$.
\end{proof}

\begin{example} \label{example: ddt but not edc}
  The logic $\ILbc$ has the global DDT but neither the local nor the parametrized EDCF.
\end{example}

\begin{proof}
  The logic $\ILbc$ has the DDT because intuitionistic logic has the DDT and the DDT is preserved under expanding the signature and adding axioms.

  To prove that $\ILbc$ does not have the local EDCF, it will suffice to show that it does not have the absolute FEP on $\Alg \ILbc$. (This implication of Theorem~\ref{thm: baby ledcf} does not depend on the assumption that $\Alg \ILbc$ is closed under subalgebras.) That is, it suffices to find algebras $\alg{A}, \alg{B} \in \Alg \ILbc$ such that $\alg{A} \leq \alg{B}$ but $\Fg^{\alg{A}}_{\logic{L}} \emptyset \neq \alg{A} \cap \Fg^{\alg{B}}_{\logic{L}} \emptyset$. The previous lemma implies that $\Fm \in \Alg \ILbc$: if $\varphi \neq \psi$, then $\pair{\varphi}{\psi} \notin \Leibniz^{\Fm} T$ for $T$ either the theory generated by $\bullet \varphi$ or the theory generated by $\bullet \psi$. Now take $\Fm'$ be the subalgebra of $\Fm$ generated by the set $\set{\circ x }{x \in \Var}$. Observe that map $x \mapsto \circ x$ for $x \in \Var$ extends to an isomorphism $i\colon \Fm \to \Fm'$. Since $\Fm \in \Alg \ILbc$, also $\Fm' \in \Alg \ILbc$.

  Because $\ILbc$ is not a trivial logic, $x \notin \Fg^{\Fm}_{\logic{L}} \emptyset$ for each $x \in \Var$. Because $i$ is an isomorphism, it follows that $\circ x \notin \Fg^{\Fm'}_{\logic{L}} \emptyset$. But $\circ x \in \Fg^{\Fm}_{\logic{L}} \emptyset$ for each $x \in \Var$ thanks to the axiom $\emptyset \vdash \circ x$. That is, $\Fm' \leq \Fm$ but $\Fg^{\Fm'}_{\logic{L}} \emptyset$ is not the restriction of $\Fg^{\Fm}_{\logic{L}} \emptyset$ to $\Fm'$.

  It remains to prove that $\ILbc$ does not have a parametrized EDCF. We prove this syntactically, by reasoning directly about the equations involved in the parametrized EDCF. Since $\Fm \in \Alg \ILbc$, suppose for the sake of contradiction that there is a set of equations $\boldTheta_{0}(x, \tuple{y})$, where $\tuple{y}$ is a possibly infinite tuple of variables, such that $\varphi \in \Fg^{\Fm}_{\logic{L}} \emptyset$ if and only if $\Fm \vDash \boldTheta_{0}(\varphi, \tuple{\alpha})$ for some tuple of formulas $\tuple{\alpha}$. Firstly, observe that over $\Fm$ each equation of the form $f(t_{1}, \dots, t_{m}) \equals g(u_{1}, \dots, u_{n})$, where $f, g$ are primitive operations and $t_{1}, \dots, t_{m}, u_{1}, \dots, u_{n}$ are terms, is either inconsistent (if $f$ and $g$ are distinct operations) or equivalent to the conjunction of $t_{1} \equals u_{1}, \dots, t_{m} \equals u_{n}$ (if $f = g$, in which case also $m = n$). Each equation in $\boldTheta_{0}(x, \tuple{y})$ is therefore either inconsistent or equivalent to a set of equations of the form $v \equals t$ where $v$ is a variable and $t$ a term. We may moreover assume that $t$ does not contain the variable $v$, otherwise the equation is either inconsistent over $\Fm$ or it has the form $v \equals v$ and as such is equivalent to the empty set of equations.

  If $v$ is one of the variables $\tuple{y}$ and $\boldTheta_{0}(x, \tuple{y})$ contains the equation $v \equals t$ with $t$ a term not containing $v$, we may remove the equation $v \equals t$ from $\boldTheta_{0}$ and substitute $v$ throughout $\boldTheta_{0}(x, \tuple{y})$ by $t$. Doing this for all such variables in $\tuple{y}$ yields a set $\boldTheta_{0}'(x, \tuple{y})$ which still satisfies the original equivalence, i.e.\ $\varphi \in \Fg^{\Fm}_{\logic{L}} \emptyset$ if and only if $\Fm \vDash \boldTheta'_{0}(\varphi, \tuple{\alpha})$ for some tuple of formulas $\tuple{\alpha}$, and moreover only contains equations of the form $x \equals t$, where $x$ is the specific variable which occurs as the first argument in $\boldTheta_{0}(x, \tuple{y})$. We may again assume that $x$ does not occur in $t$. Since $\Fm \vDash \boldTheta'_{0}(1, \tuple{\alpha})$ for some tuple of formulas $\tuple{\alpha}$, each such $t$ must be equal to $1$. But also $\Fm \vDash {\boldTheta'_{0}(x \rightarrow x, \tuple{\alpha})}$ for some $\tuple{\alpha}$, so each such $t$ must be a formula of the form $t_{1} \rightarrow t_{2}$. We have therefore obtained a contradiction.
\end{proof}

  In the opposite direction, there is no shortage of logics demonstrating that the EDCF does not imply even the parametrized local DDT (since this form of the DDT is equivalent to protoalgebraicity), for example the strong three-valued Kleene logic or the logic of order of unital meet semilattices, i.e.\ of meet semilattice with a top element $1$. A more interesting question is whether the EDCF together with protoalgebraicity implies the DDT, at least in its local or parametrized form. We now answer this question in the negative.

  Given a bounded lattice $\alg{L}$, define $\alg{L}^{+}$ as the expansion of $\alg{L}$ by the operation
\begin{align*}
  a \rightarrow b \assign \begin{cases} & 1 \text{ if } a \leq b, \\ & 0 \text{ if } a \nleq b.\end{cases}
\end{align*}
  Let $\ESL$ be the logic of order of the class $\class{K}_{\rightarrow} \assign \set{\alg{L}^{+}}{\alg{L} \text{ is a bounded lattice}}$, i.e.\ the logic determined by all matrices of the form $\pair{\alg{L}^{+}}{F}$ where $\alg{L}$ is a bounded lattice and $F$ is a lattice filter on $\alg{L}$.

\begin{example} \label{example: protoalgebraic + edcf does not imply ddt}
  $\ESL$ is a finitely equivalential logic which has an EDCF on the variety $\Alg \ESL = \AlgH \AlgS \AlgP (\class{K}_{\rightarrow})$ but not the global DDT.
\end{example}

\begin{proof}
   The class $\class{K}_{\rightarrow}$ is closed under subalgebras and ultraproducts and consists of simple algebras with a lattice reduct, so $\Alg \ESL = \AlgI \AlgS \AlgP(\class{K}) = \AlgH \AlgS \AlgP (\class{K})$ by the same argument as in the proof of Example~\ref{example: kleene has edcf}. The logic of order of $\class{K}$ coincides with the logic of order of $\AlgH \AlgS \AlgP (\class{K})$, since $\class{K} \vDash \gamma_{1} \wedge \dots \wedge \gamma_{n} \leq \varphi$ if and only if $\AlgH \AlgS \AlgP(\class{K}) \vDash \gamma_{1} \wedge \dots \wedge \gamma_{n} \leq \varphi$. The filters of $\ESL$ on $\AlgH \AlgS \AlgP (\class{K})$ are therefore precisely the non-empty lattice filters, so $\ESL$ has an EDCF on this variety.

  The logic $\ESL$ is equivalential, a set of congruence formulas being $\Delta(x, y) \assign \{ x \rightarrow y, y \rightarrow x \}$: for each lattice filter $F$ on a bounded lattice $\alg{L}$ we have $a = b$ if and only if $a \rightarrow b \in F$ and $b \rightarrow a \in F$, so $\Delta$ is a protoimplication set and
\begin{align*}
  \Delta(x_{1}, y_{1}), \Delta(x_{2}, y_{2}) \vdash_{\logic{L}} \Delta(f(x_{1}, x_{2}), f(y_{1}, y_{2}))
\end{align*}
  for each (non-constant) connective $f$ on $\ESL$, i.e.\ for $\wedge$, $\vee$, $\rightarrow$.

  To show that the logic $\ESL$ does not satisfy the global DDT, consider any non-distributive finite lattice $\alg{L}$. Then the compact filters of $\ESL$ on $\alg{L}^{+}$ form a poset isomorphic to $\alg{L}$, which is not distributive, so in particular it is not dually Brouwerian in the sense of Theorem~\ref{thm: global ddt}, therefore $\ESL$ does not have the global DDT by Theorem~\ref{thm: global ddt}.
\end{proof}

\section{Equational definition of consequence}
\label{sec: edc}

  In this final section, we return to our original motivation, namely defining logical consequence (rather than filter generation) using equations. Although we formulate no semantic characterization logics with an equational definition of consequence, we show how to prove that a logic does not have an equational definition of consequence with respect to its algebraic counterpart, using {\L}ukasiewicz logic $\Luk$ and the global modal logic $\GlobK$ as examples. This improves our previous theorems stating that these logics do not have a global EDCF. Finally, we construct a logic which has an equational definition of consequence but no global EDCF.

\begin{definition}
  A \emph{local equational definition of consequence}, or a \emph{local EDC} for short, for a logic $\logic{L}$ with respect to a class of algebras $\class{K}$ is a sequence $(\boldPsi_{n})_{n \in \omega}$ of families of sets of equations $\boldPsi_{n}(x_{1}, \dots, x_{n}, y)$, where the variables $x_{1}, \dots, x_{n}, y$ are distinct, such that
\begin{align*}
  \gamma_{1}, \dots, \gamma_{n} \vdash_{\logic{L}} \varphi \iff \class{K} \vDash \boldTheta_{n}(\gamma_{1}, \dots, \gamma_{n}, \varphi) \text{ for some } \boldTheta_{n} \in \boldPsi_{n}.
\end{align*}
  In a \emph{global equational definition of consequence} we moreover require that each $\boldPsi_{n}$ have the form $\boldPsi_{n} = \{ \boldTheta_{n} \}$ for some set of equations $\boldTheta_{n}$.
\end{definition}

  We may without loss of generality restrict to the case where $\class{K}$ is a variety in the definition of a local EDC: a local EDC only depends on the equational theory of $\class{K}$, so a local EDC with respect to $\class{K}$ is also a local EDC with respect to $\AlgH \AlgS \AlgP(\class{K})$. Less obviously, we may restrict to the case where $\class{K} \supseteq \Alg \logic{L}$.

\begin{fact} \label{fact: edct implies v l below k}
  If $\logic{L}$ has a local EDC with respect to a variety $\class{K}$, then $\Alg \logic{L} \subseteq \class{K}$.
\end{fact}

\begin{proof}
  Suppose that $\logic{L}$ has a local EDC($\boldPsi$) with respect to $\class{K}$ and $\class{K} \vDash \varphi \equals \psi$. It suffices to show that $\Alg \logic{L} \vDash \varphi \equals \psi$. Let
\begin{align*}
  \Tarski \logic{L} \assign \bigcap \set{\Leibniz^{\Fm} T}{T \text{ is a theory of } \logic{L}}.
\end{align*}
  Theorem~5.76 of~\cite{font16} states that $\AlgH \AlgS \AlgP(\Alg \logic{L}) = \AlgH \AlgS \AlgP(\Fm / \Tarski \logic{L})$, so it suffices to show that $\Fm / \Tarski \logic{L} \vDash \varphi \equals \psi$. Proposition~5.75(1) of~\cite{font16} states that to this end it suffices prove that $\pair{\varphi}{\psi} \in \Tarski \logic{L}$, i.e.\ that $\pair{\varphi}{\psi} \in \Leibniz^{\Fm} T$ for each $\logic{L}$-theory $T$. This, finally, means proving that for each formula $\chi(\tuple{x}, y)$ we have $\chi(\tuple{x}, \varphi) \in T \iff \chi(\tuple{x}, \psi) \in T$, i.e.\ that $\chi(\tuple{x}, \varphi) \vdash_{\logic{L}} \chi(\tuple{x}, \psi)$ and $\chi(\tuple{x}, \psi) \vdash_{\logic{L}} \chi(\tuple{x}, \varphi)$. By symmetry, it suffices to prove that the first rule holds in $\logic{L}$. But
\begin{align*}
  \chi(\tuple{x}, \varphi) \vdash_{\logic{L}} \chi(\tuple{x}, \psi) & \iff \class{K} \vDash \boldTheta_{1}(\chi(\tuple{x}, \varphi), \chi(\tuple{x}, \psi)) \\
  & \iff \class{K} \vDash \boldTheta_{1}(\chi(\tuple{x}, \varphi), \chi(\tuple{x}, \varphi)) \\
  & \iff \chi(\tuple{x}, \varphi) \vdash_{\logic{L}} \chi(\tuple{x}, \varphi),
\end{align*}
  where the first and last equivalences use the local EDC($\boldPsi$) and the second uses the assumption that $\class{K} \vDash \varphi \equals \psi$.
\end{proof}

  We now show that $\Luk$ and $\GlobK$ do not have an EDC. The proof for {\L}ukasiewicz logic requires a basic familiarity with the logic, which the reader might obtain from~\cite{mvbook00}, while the proof for the global modal logic $\GlobK$ requires a basic familiarity with Kripke semantics, which the reader might obtain from~\cite{blackburn+derijke+venema01}.

  These proofs follow the same pattern: we start with a particular rule such that $\varphi \nvdash_{\logic{L}} \psi$, so $\Alg \logic{L} \nvDash E(\varphi, \psi)$ for $E(x, y) \assign \boldTheta_{1}(x, y)$ by the standard EDC. In both cases $\Alg \logic{L}$ is a variety generated by some manageable class of algebras~$\class{K}$. In the case of {\L}ukasiewicz logic $\class{K}$ is the singleton class consisting of the standard MV-chain $[0, 1]$, while in the case of modal logic it is the class of complex algebras of Kripke frames. This yields an algebra $\alg{A} \in \class{K}$ such that $\alg{A} \nvDash E(\varphi, \psi)$. We then find a valuation on $\alg{A}$ witnessing that $\alg{A} \nvDash E(\varphi', \psi')$ for some formulas $\varphi'$ and $\psi'$ with $\varphi' \vdash_{\logic{L}} \psi'$, which contradicts the standard EDC.

\begin{example} \label{example: luk does not have edc}
  {\L}ukasiewicz logic $\Luk$ does not have a standard EDC.
\end{example}

\begin{proof}
  Suppose for the sake of contradiction that $\Luk$ has a standard EDC, in particular there is a set of equations in two variables $E(x, y)$ such that $\varphi \vdash_{\Luk} \psi$ if and only if $\MV \vDash E(\varphi, \psi)$. Then $\MV \nvDash \varepsilon(1, 0)$ for some equation $\varepsilon \in E$, since $1 \nvdash_{\Luk} 0$. Because $\MV$ is generated as a variety by $[0, 1]$, it follows that $[0, 1] \nvDash \varepsilon(1, 0)$. Since each primitive operation of $[0, 1]$, hence also each term operation of $[0, 1]$, is continuous with respect to the standard Euclidean topology, the solution set $\set{\pair{a}{b} \in [0, 1]^{2}}{\MV \vDash \varepsilon(a, b)}$ is closed in this topology. Consequently, there is some $a < 1$ such that $[0, 1] \nvDash \varepsilon(a, 0)$. But then $a^{n} = 0$ for some $n \in \omega$, so $[0, 1] \nvDash \varepsilon (a, a^{n})$ and $\MV \nvDash E(x, x^{n})$, contradicting the fact that $x \vdash_{\Luk} x^{n}$.
\end{proof}

\begin{example} \label{example: global k does not have edc}
  The global modal logic $\GlobK$ does not have a standard EDC.
\end{example}

\begin{proof}
  Suppose for the sake of contradiction that $\GlobK$ has a standard EDC. In particular, there is a set of equations in two variables $E(x, y)$ such that $\varphi \vdash_{\GlobK} \psi$ if and only if $\BAO \vDash E(\varphi, \psi)$. Because $\Diamond 1 \nvdash_{\GlobK} 0$, there is some equation $\varepsilon \in E$ such that $\BAO \nvDash \varepsilon(\Diamond 1, 0)$. We may assume without loss of generality that $\varepsilon(x, y)$ has the form $\varphi(x, y) \equals 1$ for some formula in two variables $\varphi(x, y)$. Because $\GlobK$ is Kripke complete with respect to the class of trees (irreflexive directed trees where every edge points away from the root), there is a valuation $\Vdash$ on a tree $T$ with root $r$ such that $r \nVdash \varphi(\Diamond 1, 0)$. Since $\varphi(\Diamond 1, 0)$ does not contain any variables, $r \nVdash \varphi(\Diamond 1, 0)$ in each valuation $\Vdash$ on $T$.

  Now observe that $\varphi$ only contains nestings of modal operators of some limited depth, so whether $r \Vdash \varphi(\alpha, \beta)$ for a given valuation $\Vdash$ on $T$ only depends on the values of $\alpha$ and $\beta$ at worlds of limited height. More precisely, there is some $n \in \omega$ such that for each valuation $\Vdash$ on $T$ we have $r \nVdash \varphi(\alpha, \beta)$ whenever $w \Vdash \alpha \iff w \Vdash \Diamond 1$ and $w \Vdash \beta \iff w \Vdash 0$ for each world $w$ in $T$ of height at most $n$ (assigning height $0$ to the root $r$).

  Consider the valuation $\Vdash$ such that $w \Vdash x$ if and only if $w \Vdash \Diamond 1$ and moreover $w$ has height at most $n$. Then $r \nVdash \varphi(x, 0)$. Moreover, $w \nVdash \Box_{n+1} x$ for each $w$ of height at most $n$, since from each such world $w$ one can access in at most $n+1$ steps either a blind world (a world where $\Diamond 1$ fails) or a world of height strictly more than $n$. (Recall that $\Box_{0} x \assign x$ and $\Box_{i+1} x \assign x \wedge \Box \Box_{i} x$ for $i \in \omega$.) Thus $r \nVdash \varphi(x, \Box_{n+1} x)$. Switching back to algebraic semantics, this means that $\BAO \nvDash \varphi(x, \Box_{n+1} x) \equals 1$, so $\BAO \nvDash E(x, \Box_{n+1} x)$. But the EDC now implies that $x \nvdash_{\GlobK} \Box_{n+1} x$, whereas in fact $x \vdash_{\GlobK} \Box_{n+1} x$.
\end{proof}

  Finally, we construct a logic which has a standard EDC but not an EDCF. Consider a variety of algebras $\class{K}$ with a constant $1$ in their signature. Let $\logic{L}_{\class{K}}$ be the logic of all matrices of the form $\pair{\alg{A}}{F}$ such that $\alg{A} \in \class{K}$ and $1 \in F$. We have $\Alg \logic{L}_{\class{K}} = \class{K}$, since clearly $\class{K} \subseteq \Alg \logic{L}$ and conversely $\Alg \logic{L} \subseteq \class{K}$ by Fact~\ref{fact: alg l variety}. The logic $\logic{L}_{\class{K}}$ can be described more explicitly as follows:
\begin{align*}
  \gamma_{1}, \dots, \gamma_{n} \vdash_{\logic{L}} \varphi & \iff \class{K} \vDash \gamma_{i} \equals \varphi \text{ for some } i \in \{ 1, \dots, n \} \text{ or } \class{K} \vDash \varphi \equals 1.
\end{align*}
  In fact, this local EDC extends to a local EDCF: for each $\alg{A} \in \class{K}$
\begin{align*}
  b \in \Fg^{\alg{A}}_{\logic{L}_{\class{K}}} (a_{1}, \dots, a_{n}) \iff b = a_{i} \text{ for some } i \in \{ 1, \dots, n \} \text{ or } b = 1.
\end{align*}

  We now take a particular variety $\class{K}$ such that $\logic{L}_{\class{K}}$ has a global EDC but no global EDCF. Namely, the signature of $\class{K}$ shall consist of a constant $1$ and an $(n+1)$-ary function symbol $\Box_{n}$ for each $n \geq 1$, and $\class{K}$ shall be the variety in this signature axiomatized by the following sequence of equations for $k \geq 1$:
\begin{align*}
  \Box_{k}(x_{1}, \dots, x_{k}, x_{i}) \equals 1. \tag{$\alpha_{k}$}
\end{align*}

\begin{lemma}
  ~
\begin{enumerate}[(i)]
\item $\class{K} \vDash \Box_{m}(t_{1}, \dots, t_{m}, u) \equals 1$ \hskip -0.51pt if and only if $\class{K} \vDash t_{i} \equals u$ \hskip -0.51pt for some ${i \in \{ 1, \dots, m \}}$.
\item $\class{K} \vDash 1 \equals \Box_{m}(t_{1}, \dots, t_{m}, u)$ \hskip -0.51pt if and only if $\class{K} \vDash t_{i} \equals u$ \hskip -0.51pt for some ${i \in \{ 1, \dots, m \}}$.
\item $\class{K} \vDash \Box_{m}(t_{1}, \dots, t_{m}, u) \equals \Box_{n}(v_{1}, \dots, v_{n}, w)$ if and only if either (a) there are $i \in \{ 1, \dots, m \}$ and $j \in \{ 1, \dots, n \}$ such that $\class{K} \vDash t_{i} \equals u$ and $\class{K} \vDash v_{j} \equals w$, or (b) $m = n$ and $\class{K} \vDash t_{i} \equals v_{i}$ for each $i \in \{ 1, \dots, m \}$ and $\class{K} \vDash u \equals w$.
\end{enumerate}
\end{lemma}

\begin{proof}
  The right-to-left implications are trivial consequences of the axioms $(\alpha_{k})$. We prove the left-to-right implications by induction over the depth of the equational proof of the equality. The equational calculus that we shall use is a variant of the calculus presented in~\cite[Definition~II.14.16]{burris+sankappanavar81}. Axioms of this calculus are either instances of Reflexivity, i.e.\ equalities of the form $t \equals t$ for some term $t$, or substitution instances of one of the axioms $(\alpha_{k})$, or the symmetric converses of such substitution instances, i.e.\ equalities of the form $1 \equals \Box_{k}(t_{1}, \dots, t_{n}, u)$. The calculus has two inference rules. Transitivity allows us to infer $t \equals v$ from $t \equals u$ and $u \equals v$. Replacement allows us to infer $t(p, y_{1}, \dots, y_{n}) \equals t(q, y_{1}, \dots, y_{n})$ from $p \equals q$ for each term $t(x, y_{1}, \dots, y_{n})$ in the variables $x, y_{1}, \dots, y_{n}$. We may assume that Replacement is only used non-trivially, i.e.\ that it is not applied to the terms $t \assign 1$ and $t \assign x$, where $x$ is a variable. (The reader can easily verify that the rules of Symmetry and Substitution in the calculus of~\cite{burris+sankappanavar81} are admissible in this variant of the calculus, in the sense that any equality provable using these rules can be proved without them.)

  We now proceed by induction over the depth of the equational proof of the three types of equalities. The equality $\Box_{m}(t_{1}, \dots, t_{m}, u) \equals 1$ is neither an instance of Reflexivity nor an instance of the symmetric converse of $(\alpha_{k})$ for any ${k \geq 1}$. If it is an instance of $(\alpha_{k})$ for some ${k \geq 1}$, then it must be an instance of $(\alpha_{m})$ and consequently $t_{i} = u$ for some $i \in \{ 1, \dots, m \}$ as required. The equality $\Box_{m}(t_{1}, \dots, t_{m}, u) \equals 1$ cannot be the result of a non-trivial application of Replacement.

  Suppose that $\Box_{m}(t_{1}, \dots, t_{m}, u) \equals 1$ is the result of an application of Transitivity with premises $\Box_{m}(t_{1}, \dots, t_{m}, u) \equals v$ and $v \equals 1$. If $v$ is the term $1$, then the claim holds by the inductive hypothesis. Otherwise, $v$ has the form $\Box_{n}(p_{1}, \dots, p_{n}, q)$. Then the inductive hypothesis applied to $\Box_{m}(t_{1}, \dots, t_{m}, u) \equals \Box_{n}(p_{1}, \dots, p_{n}, q)$ yields that either $\class{K} \vDash t_{i} \equals u$ for some $i \in \{ 1, \dots, m \}$ as required, or $m = n$ and $\class{K} \vDash t_{i} \equals p_{i}$ for each $i \in \{ 1, \dots, m \}$ and $\class{K} \vDash q \equals u$. But in the latter case the inductive hypothesis applied to $\Box_{n}(p_{1}, \dots, p_{n}, q) \equals 1$ yields $\class{K} \vDash p_{i} \equals q$ for some $i \in \{ 1, \dots, n \}$, so $\class{K} \vDash t_{i} \equals p_{i} \equals q \equals u$, as required.

  The equality $\Box_{m}(t_{1}, \dots, t_{m}, u) \equals \Box_{n}(p_{1},\dots, p_{n}, q)$ is only an instance of Reflexivity if $m = n$ and $t_{i} = p_{i}$ for $i \in \{ 1, \dots, n \}$ and $u = q$, as required. It is not an instance of $(\alpha_{k})$ or the symmetric converse of $(\alpha_{k})$. The equality can only be the result of a non-trivial application of Replacement if $m = n$ and $t_{j} \equals p_{j}$ is the result of Replacement for some $j \in \{ 1, \dots, n \}$ and $t_{i} = p$

  Suppose that $\Box_{m}(t_{1}, \dots, t_{m}, u) \equals \Box_{n}(p_{1},\dots, p_{n}, q)$ is the result of an application of Transitivity. If the premises are $\Box_{m}(t_{1}, \dots, t_{m}, u) \equals 1$ and $1 \equals \Box_{n}(p_{1},\dots, p_{n}, q)$, then by the inductive hypothesis $\class{K} \vDash t_{i} \equals u$ and $\class{K} \vDash p_{j} \equals q$ for some $i \in \{ 1, \dots, m \}$ and $j \in \{ 1, \dots, n \}$, as required. If the premises are $\Box_{m}(t_{1}, \dots, t_{m}, u) \equals \Box_{k}(r_{1},\dots, r_{k}, s)$ and $\Box_{k}(r_{1}, \dots, r_{k}, s) \equals \Box_{n}(p_{1}, \dots, p_{n}, q)$, then there are four cases:
\begin{enumerate}[(i)]
\item There are $\class{K} \vDash t_{i} \equals u$ and $\class{K} \vDash p_{j} \equals q$ for some $i \in \{ 1, \dots, m \}$and $j \in \{ 1,\dots, n \}$ as required.
\item We have $m = k$ and $\class{K} \vDash t_{i} \equals r_{i}$ for each $i \in \{ 1, \dots, m \}$ and $\class{K} \vDash u \equals s$, and moreover $\class{K} \vDash r_{j} \equals s$ and $\class{K} \vDash p_{l} \equals q$ for some $j \in \{ 1, \dots, k \}$ and $l \in \{ 1, \dots, n \}$. Then $\class{K} \vDash t_{j} \equals r_{j} \equals s \equals u$ and $\class{K} \vDash p_{l} \equals q$ as required.
\item We have $\class{K} \vDash t_{i} \equals u$ and $\class{K} \vDash r_{j} \equals s$ for some $i \in \{ 1, \dots, m \}$ and $j \in \{ 1, \dots, n \}$, and moreover $k = n$ and $\class{K} \vDash r_{j} \equals p_{j}$ for each $j \in \{ 1, \dots, k \}$ and $\class{K} \vDash s \equals q$. Then $\class{K} \vDash t_{i} \equals u$ and $\class{K} \vDash p_{j} \equals r_{j} \equals s \equals q$ as required.
\item We have $m = k = n$ and $\class{K} \vDash t_{i} \equals r_{i} \equals p_{i}$ for each $i \in \{ 1, \dots, m \}$ and $\class{K} \equals u \equals s \equals q$ as required, 
\end{enumerate}
  Finally, equalities of the form $1 \equals \Box_{m}(t_{1}, \dots, t_{m}, u)$ are dealt with in the same way as equalities of the form $\Box_{m}(t_{1}, \dots, t_{m}, u) \equals 1$.
\end{proof}

\begin{example} \label{example: edc without edcf}
  The logic $\logic{L}_{\class{K}}$ has a standard EDC with respect to $\class{K}$:
\begin{align*}
  \gamma_{1}, \dots, \gamma_{n} \vdash_{\logic{L}_{\class{K}}} \varphi & \iff \class{K} \vDash \Box_{n+1}(1, \gamma_{1}, \dots, \gamma_{n}, \varphi).
\end{align*}
  It has a local EDCF but no global EDCF.
\end{example}

\begin{proof}
  The equivalence follows immediately from the previous lemma. Clearly $\class{K} \vDash \varphi \equals \psi$ implies $\varphi \dashv \vdash_{\logic{L}_{\class{K}}} \psi$. Conversely, if $\varphi \vdash_{\logic{L}_{\class{K}}} \psi$ and $\psi \vdash_{\logic{L}_{\class{K}}} \varphi$, then either at least one of the rules holds by virtue of the equality $\class{K} \vDash \varphi \equals \psi$ or they hold by virtue of the equalities $\class{K} \vDash \psi \equals 1$ and $\class{K} \vDash \varphi \equals 1$, in which case we again obtain $\class{K} \vDash \varphi \equals \psi$. Thus $\varphi \dashv \vdash_{\logic{L}_{\class{K}}} \psi$ if and only if $\class{K} \vDash \varphi \equals \psi$. By~\cite[Proposition~7.8]{font16} it follows that $\class{K}$ is the intrinsic variety of $\logic{L}_{\class{K}}$, so $\logic{L}_{\class{K}}$ enjoys a standard EDC.

  We now show that $\logic{L}_{\class{K}}$ does not have a global EDCF. To this end, consider the algebra $\alg{A} \in \class{K}$ over the underlying set $\{ 0, 1, a_{1}, a_{2}, b \}$ such that $\Box_{n}(x_{1}, \dots, x_{n}, y) = 1$ if $y \in \{ x_{1}, \dots, x_{n} \}$ else $\Box_{n}(x_{1}, \dots, x_{n}, y) = 0$. We show that there is no smallest $\class{K}$-congruence $\theta$ on $\alg{A}$ such that $b / \theta \in \Fg^{\alg{A} / \theta}_{\logic{L}_{\class{K}}} (a_{1}, a_{2}) / \theta$. To see this, let $\theta_{1}$ and $\theta_{2}$ be the equivalence relations whose only non-trivial equivalence classes are $\{ 0, 1 \}$ and, respectively, $\{ a_{1}, b \}$ and $\{ a_{2}, b \}$. These are congruences, and thus $\class{K}$-congruences: to see this, it suffices to observe that replacing any argument of $\Box_{n}(x_{1}, \dots, x_{n}, y)$ by an equivalent one will at most result in changing the value from $1$ to $0$ or vice versa, and $\pair{0}{1} \in \theta_{1} \cap \theta_{2}$. But for $\theta \assign \theta_{1} \cap \theta_{2}$ it is not the case that $b / \theta \in \Fg^{\alg{A} / \theta}_{\logic{L}_{\class{K}}} (a_{1}, a_{2}) / \theta$.
\end{proof}

\end{document}